\documentclass{amsart}
\usepackage[utf8]{inputenc}
\usepackage{amssymb,hyperref,mathrsfs, mathtools,booktabs}
\usepackage[all]{xypic}
\usepackage[lmargin=1in,rmargin=1in,tmargin=1in,bmargin=1in]{geometry}
\usepackage{tikz}
\usetikzlibrary{cd}
\usepackage{bm}
\usetikzlibrary{arrows}
\usepackage{enumitem}
\tikzset{aar/.style={->, thick}}
\tikzset{taar/.style={double, double equal sign distance, -implies}}
\tikzset{amar/.style={->, dotted}}
\tikzset{dmar/.style={->, dashed}}
\newcommand{\lab}[1]{$\scriptstyle #1$}

\newcommand{\co}{\nobreak\mskip2mu\mathpunct{}\nonscript
  \mkern-\thinmuskip{:}\penalty300\mskip6muplus1mu\relax}
\renewcommand{\th}{^{\text{th}}}
\newcommand{\lsub}[2]{{}_{#1}#2}
\newcommand{\lsup}[2]{{}^{#1}\mskip-.6\thinmuskip#2}

\def\mathcenter#1{%
  \vcenter{\hbox{$#1$}}%
}

\newcommand{\ZZ}{\mathbb{Z}}

\newcommand{\FF}{\mathbb{F}}
\newcommand{\Field}{\FF_2}
\newcommand{\op}{\mathrm{op}}
\newcommand{\Ainf}{A_\infty}
\newcommand{\Pin}{\mathit{Pin}}
\newcommand{\Spin}{\mathit{Spin}}
\newcommand{\bdy}{\partial}
\DeclareMathOperator{\gr}{gr}
\DeclareMathOperator{\Cone}{Cone}
\newcommand{\dg}{\textit{dg }}
\newcommand{\Id}{\mathrm{Id}}
\newcommand{\bId}{\mathbb{I}}
\DeclareMathOperator{\coker}{coker}
\DeclareMathOperator{\Mor}{Mor}
\DeclareMathOperator{\Hom}{Hom}
\DeclareMathOperator{\rank}{rank}

\newcommand{\HF}{\mathit{HF}}
\newcommand{\HFa}{\widehat{\mathit{HF}}}
\newcommand{\CF}{{\mathit{CF}}}
\newcommand{\CFa}{\widehat{\mathit{CF}}}

\newcommand{\Alg}{\mathcal{A}}

\newcommand{\Blg}{\mathcal{B}}
\newcommand{\CFD}{\mathit{CFD}}
\newcommand{\CFDD}{\mathit{CFDD}}
\newcommand{\CFA}{\mathit{CFA}}

\newcommand{\CFDA}{\mathit{CFDA}}
\newcommand{\CFDAa}{\widehat{\CFDA}}
\newcommand{\CFAA}{\mathit{CFAA}}
\newcommand{\CFAAa}{\widehat{\CFAA}}
\newcommand{\CFDa}{\widehat{\CFD}}
\newcommand{\CFAa}{\widehat{\CFA}}
\newcommand{\CFDDa}{\widehat{\CFDD}}
\newcommand{\CFDI}{\mathit{CFDI}}
\newcommand{\CFAI}{\mathit{CFAI}}
\newcommand{\CFDIa}{\widehat{\CFDI}}
\newcommand{\CFAIa}{\widehat{\CFAI}}
\newcommand{\DD}{$\mathit{DD}$}
\newcommand{\DA}{$\mathit{DA}$}
\newcommand{\AAm}{$\mathit{AA}$}
\newcommand{\DT}{\boxtimes}
\newcommand{\CFI}{\mathit{CFI}}
\newcommand{\CFIa}{\widehat{\CFI}}
\newcommand{\HFI}{\mathit{HFI}}
\newcommand{\HFIm}{\mathit{HFI}^{-}}
\newcommand{\HFIa}{\widehat{\HFI}}
\newcommand{\conj}[1]{\overline{\overline{#1}}}
\newcommand{\bunderline}[1]{\underline{#1\mkern-2mu}\mkern2mu }

\def\du {\bar{d}}
\def\dl {\bunderline{d}}
\newcommand{\HD}{\mathcal{H}}
\newcommand{\alphas}{\boldsymbol{\alpha}}
\newcommand{\betas}{\boldsymbol{\beta}}

\newcommand{\x}{\mathbf{x}}

\newcommand{\HB}{H}
\newcommand{\spinc}{\mathfrak{s}}
\newcommand{\SpinC}{\mathrm{spin}^c}
\newcommand{\spinC}{\mathit{Spin}^c}
\renewcommand{\Spin}{\mathrm{spin}}
\newcommand{\Gens}{\mathfrak{S}}
\newcommand{\PMC}{\mathcal{Z}}
\newcommand{\AZ}{\mathsf{AZ}}
\newcommand{\bAZ}{\overline{\AZ}}
\newcommand{\AZbar}{\bAZ}
\newcommand{\Chord}{\mathrm{chord}}
\DeclareMathOperator{\Sym}{Sym}

\newcommand{\ModCat}{\mathsf{Mod}}

\newcommand{\MCG}{\mathit{MCG}}
\newcommand{\nbd}{\mathrm{nbd}}

\theoremstyle{plain}

\numberwithin{equation}{section}
\newtheorem{theorem}[equation]{Theorem}

\newtheorem{proposition}[equation]{Proposition}
\newtheorem{lemma}[equation]{Lemma}
\newtheorem{corollary}[equation]{Corollary}

\newtheorem{conjecture}[equation]{Conjecture}

\newtheorem{convention}[equation]{Convention}
\newtheorem{definition}[equation]{Definition}

\theoremstyle{definition}

\theoremstyle{remark}
\newtheorem{example}[equation]{Example}
\newtheorem{remark}[equation]{Remark}

\definecolor{darkgreen}{rgb}{0,.25,0}
\definecolor{darkred}{rgb}{.25,0,0}
\newcommand{\greenit}[1]{\textcolor{darkgreen}{#1}}
\newcommand{\orangeit}[1]{\textcolor{darkred}{#1}}

\makeatletter
\providecommand\@dotsep{5}
\def\listtodoname{List of Todos}
\def\listoftodos{\@starttoc{tdo}\listtodoname}
\makeatother

\copyrightinfo{2018}{Kristen Hendricks and Robert Lipshitz}

\begin{document}

\title{Involutive bordered Floer homology}

\author{Kristen Hendricks}
 \address{Mathematics Department, Michigan State University\\
   East Lansing, MI 48824}
   \thanks{\texttt{KH was supported by NSF Grant DMS-1663778.}}
\email{\href{mailto:hendricks@math.msu.edu}{hendricks@math.msu.edu}}

\author{Robert Lipshitz}
\thanks{\texttt{RL was supported by NSF Grant DMS-1642067.}}
\address{Department of Mathematics, University of Oregon, Eugene, OR 97403}
\email{\href{mailto:lipshitz@uoregon.edu}{lipshitz@uoregon.edu}}

\subjclass[2010]{Primary 57R58, Secondary 57M27}

\date{}

\dedicatory{}

\begin{abstract}
  	We give a bordered extension of involutive $\HFa$ and use it to give an algorithm to compute involutive $\HFa$ for general $3$-manifolds. We also explain how the mapping class group action on $\HFa$ can be computed using bordered Floer homology. As applications, we prove that involutive $\HFa$ satisfies a surgery exact triangle and compute $\HFIa(\Sigma(K))$ for all 10-crossing knots $K$.
\end{abstract}

\maketitle


\tableofcontents

\section{Introduction}\label{sec:intro}
In 2013, Manolescu introduced a $\Pin(2)$-equivariant version of Seiberg-Witten Floer homology and used it to resolve the Triangulation Conjecture~\cite{Manolescu16:triangulation}. Since then, several authors have given applications of these new invariants, particularly to the homology cobordism group~\cite{Manolescu14:intersection,Lin15:KO,Stoffregen:SFS,Stoffregen:sum,Stoffregen:remark}.
F.~Lin also gave a reformulation of $\Pin(2)$-equivariant
Seiberg-Witten Floer homology, and deduced a number of formal properties, such as a surgery exact triangle, in addition to various applications~\cite{Lin:MB-Floer,Lin:Exact,Lin:higher-comp,Lin:correction,Lin:involutive-Kh}.

Two years later, Manolescu and the first author introduced a shadow of
$\Pin(2)$-equivariant Seiberg-Witten Floer homology, called
\emph{involutive Heegaard Floer homology}~\cite{HM:involutive}, in
Ozsv\-\'ath-Szab\'o's Heegaard Floer
homology~\cite{OS04:HolomorphicDisks}. Involutive Heegaard Floer
homology has also had a number of applications, again mainly to the homology cobordism group~\cite{HMZ:involutive-sum,BH:cuspidal,DM:involutive-plumbed, Zemke:KnotSum}. 

As described below, a key step in the definition of involutive Heegaard Floer homology is naturality of the Heegaard Floer invariants~\cite{OS06:HolDiskFour,JT:Naturality}. Another implication of naturality is that the mapping class group of a $3$-manifold $Y$ acts on the Heegaard Floer invariants of $Y$; this action has been studied relatively little.

Bordered Heegaard Floer homology, introduced by Ozsv\'ath, Thurston,
and the second author, extends the Heegaard Floer invariant $\HFa(Y)$
to $3$-manifolds with boundary~\cite{LOT1,LOT2}, leading to a practical algorithm for computing $\HFa(Y)$~\cite{LOT4}. 
In this paper we extend that algorithm to compute both the hat variant
of involutive Heegaard Floer homology and the mapping class group
action on $\HFa(Y)$. Although the two actions are different, their
description in terms of bordered Floer homology are quite similar. We also prove a (hitherto unknown) surgery exact triangle for the hat variant of involutive Heegaard Floer homology. In the rest of the introduction we recall some of the definitions and sketch how these algorithms work.

Given a 3-manifold $Y$, the minus (respectively hat) involutive Heegaard Floer complex of $Y$ is defined as follows~\cite{HM:involutive}. Fix a pointed Heegaard diagram $\HD$ for $Y$. Recall that $\CF^-(\HD)$ (respectively $\CFa(\HD)$) is a chain complex of free $\Field[U]$-modules (respectively $\Field$-vector spaces). Consider the modules
\begin{align*}
\CFI^-(\HD)&=\CF^-(\HD)[-1]\otimes_{\Field[U]}\Field[U,Q]/(Q^2)\\
\CFIa(\HD)&=\CFa(\HD)[-1]\otimes_{\Field}\Field[Q]/(Q^2),
\end{align*}
over $\Field[U,Q]/(Q^2)$ (respectively $\Field[Q]/(Q^2)$), where $Q$ has degree $-1$ and $U$ has degree $-2$. Define a differential on $\CFI^-(\HD)$ and $\CFIa(\HD)$ by 
\begin{equation}\label{eq:CFI-diff}
\bdy_{\CFI}(x)=\bdy_{\CF}(x)+[x+\iota(x)]Q,
\end{equation}
where $\bdy_{\CF}(x)$ is the usual differential on $\CF^-(\HD)$ or $\CFa(\HD)$ and $\iota$ is an endomorphism of $\CF^-(\HD)$ or $\CFa(\HD)$ defined as follows. Let $\conj{\HD}$ be the result of exchanging the roles of the $\alpha$- and $\beta$-circles and reversing the orientation of the Heegaard surface; i.e., if
\begin{align*}
\HD&=(\Sigma,\alphas,\betas,z)\\
\intertext{then}
\conj{\HD}&=(-\Sigma,\betas,\alphas,z).
\end{align*}

Given a generator $\x=\{x_i\in\alpha_i\cap
\beta_{\sigma(i)}\}\subset\Sigma$ for $\CF^-(\HD)$ (respectively
$\CFa(\HD)$), exactly the same set of points gives a generator
$\eta(\x)$ for $\CF^-(\conj{\HD})$. For suitable choices of almost complex structures on $\Sym^g(\Sigma)$ and $\Sym^g(-\Sigma)$, the map $\eta$ is a chain isomorphism. Next, since $\HD$ and $\conj{\HD}$ both represent $Y$, there is a sequence of Heegaard moves from $\conj{\HD}$ to $\HD$. There is then a corresponding chain homotopy equivalence $\Phi\co \CF^-(\conj{\HD})\to\CF^-(\HD)$ (respectively $\Phi\co\CFa(\conj{\HD})\to\CFa(\HD)$) associated to this sequence of Heegaard moves (together with changes of almost complex structures)~\cite{OS04:HolomorphicDisks}; the map $\Phi$ is well-defined up to chain homotopy~\cite{OS06:HolDiskFour,JT:Naturality,HM:involutive}. Then
\[
  \iota=\Phi\circ\eta.
\]
Formula~\eqref{eq:CFI-diff} makes $\CFI^-(\HD)$ (respectively $\CFIa(\HD)$) into a differential $\Field[U,Q]/(Q^2)$-module (respectively $\Field[Q]/(Q^2)$-module), and hence the homology $\HFI^-(\HD)$ (respectively $\HFIa(\HD)$) is also a module over $\Field[U,Q]/(Q^2)$ (respectively $\Field[Q]/(Q^2)$).

In this paper, we will focus mainly on $\CFIa(\HD)$ and its homology $\HFIa(\HD)$. Up to isomorphism, the homology groups $\HFIa(\HD)$ are determined by the induced map $\iota_*\co \HFa(\HD)\to\HFa(\HD)$ on homology:
\[
\HFIa(\HD)\cong(\ker(\Id+\iota_*)\oplus Q\coker(\Id+\iota_*))[-1]
\]
with the obvious $\Field[Q]/(Q^2)$-module structure (e.g., if $x\in\ker(\Id+\iota_*)\subset \HFa(\HD)$ then $Qx$ is the image of $x$ in $\coker(\Id+\iota_*)$).

Before explaining how to compute involutive Heegaard Floer homology, we review the bordered algorithm to compute $\HFa(Y)$~\cite{LOT4}. (This was not the first algorithm to compute $\HFa(Y)$, which was discovered by Sarkar-Wang~\cite{SarkarWang07:ComputingHFhat}.) Choosing a Heegaard splitting of $Y$ allows us to write $Y$ as a union of two (standard) handlebodies $H_g$ of genus $g$, glued by a diffeomorphism $\psi\co \Sigma_g\to\Sigma_g$ of their boundaries. Let $\PMC$ be the split, genus $g$ pointed matched circle~\cite[Figure 4]{LOT4}, and $F(\PMC)$ the corresponding surface. 
Let $\phi_0\co F(\PMC)\to \bdy H_g$ be the $0$-framed parametrization~\cite[Section 1.4.1]{LOT4}. Then 
\begin{equation}\label{eq:compute-CFa}
  \CFa(Y)\simeq \CFAa(H_g,\phi_0)\DT_{\Alg(\PMC)}\CFDa(H_g,\phi_0\circ\psi).
\end{equation}
The bordered modules $\CFAa(H_g,\phi_0)$ and $\CFDa(H_g,\phi_0)$ can be described explicitly; see Section~\ref{sec:comp-hb}. Further, if $\CFDAa(\psi)$ is the type \DA\ bordered bimodule associated to the mapping cylinder of $\psi$ then
\[
  \CFDa(H_g,\phi_0\circ\psi)\simeq \CFDAa(\psi)\DT_{\Alg(\PMC)}\CFDa(H_g,\phi_0).
\]
One factors $\psi$ as a composition $\psi=\psi_1\circ\dots\circ\psi_n$ where each $\psi_i$ is an arcslide~\cite[Section 2.1]{LOT4}. Then
\[
  \CFDAa(\psi)\simeq \CFDAa(\psi_1)\DT\cdots\DT\CFDAa(\psi_n).
\]
The type \DD\ bimodule $\CFDDa(\psi_i)$ associated to each arcslide can be described explicitly~\cite[Section 4]{LOT4}. The type \DA\ bimodule $\CFDAa(\psi_i)$ can be computed as
\begin{align*}
  \CFDAa(\psi_i)&\simeq \CFAAa(\bId)\DT_{\Alg(-\PMC)}\CFDDa(\psi_i),\\
  \CFAAa(\bId)&\simeq \CFAAa(\AZ\cup\bAZ),
\end{align*}
and $\AZ\cup\bAZ$ is a particular nice bordered Heegaard diagram
introduced by Auroux and
Zarev~\cite{Auroux10:Bordered,Zarev:JoinGlue,LOTHomPair} (see
Section~\ref{sec:AZ}), whose type \AAm\ bimodule is, consequently,
easy to describe.

Combining these steps gives an algorithm to compute $\CFa(Y)$. This algorithm is practical, at least for manifolds with small Heegaard genus and not-too-complicated gluing maps~\cite[Section 9.5]{LOT4}. Further improvements have been made by Zhan~\cite{Zhan14:thesis}.

The other key tools for computing involutive Heegaard Floer homology come from
earlier work on dualities in bordered Heegaard Floer
homology~\cite{LOTHomPair}. Recall that a bordered Heegaard diagram
consists of an oriented surface-with-boundary $\Sigma$, a collection
$\alphas$ of arcs and circles in $\Sigma$, a collection $\betas$ of
circles in $\Sigma$, and a basepoint $z$ in $\bdy\Sigma$ satisfying
certain conditions~\cite[Section 4.1]{LOT1}. We can also consider a
\emph{$\beta$-bordered} Heegaard diagram, in which $\alphas$ consists
only of circles and $\betas$ consists of arcs and
circles~\cite[Section 3.1]{LOTHomPair}. Given a bordered Heegaard diagram $\HD$,
there is an associated $\beta$-bordered Heegaard diagram $\HD^\beta$,
obtained by exchanging the roles of the $\alpha$- and $\beta$-curves
in $\HD$. The boundary of a $\beta$-bordered Heegaard diagram is a
\emph{$\beta$-pointed matched circle}. Given a pointed matched circle
$\PMC$, let $\PMC^\beta$ be the corresponding $\beta$-pointed matched
circle. Another operation on bordered Heegaard diagrams (respectively
pointed matched circles) is reversal of the orientation of the
Heegaard surface (respectively circle); we will denote this with a
minus sign. Given a Heegaard diagram $\HD$ with boundary $\PMC$,
the invariants of these objects are related as follows:
\begin{align*}
  \Alg(\PMC^\beta)&=\Alg(\PMC)^\op=\Alg(-\PMC)\\
  \lsup{\Alg(-\PMC^\beta)}\CFDa(\HD^\beta)&=\lsup{\Alg(\PMC)}\CFDa(\HD^\beta)\cong\overline{\lsup{\Alg(-\PMC)}\CFDa(\HD)}\\
  \lsup{\Alg(\PMC)}\CFDa(-\HD)&=\CFDa(-\HD)^{\Alg(\PMC)}\cong\overline{\lsup{\Alg(-\PMC)}\CFDa(\HD)}\\
  \CFAa(\HD^\beta)_{\Alg(\PMC^\beta)}&=\CFAa(\HD^\beta)_{\Alg(-\PMC)}\cong\overline{\CFAa(\HD)_{\Alg(\PMC)}}\\
  \CFAa(-\HD)_{\Alg(-\PMC)}&=\lsub{\Alg(\PMC)}\CFAa(-\HD)\cong\overline{\CFAa(\HD)_{\Alg(\PMC)}},
\end{align*}
where the overline denotes the dual $\Ainf$-module or type
$D$ structure~\cite{LOTHomPair}. As usual in the bordered Floer
literature, we are using superscripts to denote type $D$ structures
and subscripts for $\Ainf$ actions.

Given a bordered Heegaard diagram $\HD$ with boundary $\PMC$, let
$\conj{\HD}=-\HD^\beta$, so $\conj{\HD}$ is a $\beta$-bordered
Heegaard diagram with boundary $\conj{\PMC}=-\PMC^\beta$. From the
isomorphisms above, it follows that:
\begin{align*}
  \lsup{\Alg(-\PMC)}\CFDa(\conj{\HD})&\cong\lsup{\Alg(-\PMC)}\CFDa(\HD) &
  \CFAa(\conj{\HD})_{\Alg(\PMC)}&\cong \CFAa(\HD)_{\Alg(\PMC)}
\end{align*}
These are the analogues of the isomorphism $\eta$ in the definition of
$\CFI$, and we will denote these isomorphisms by $\eta$ as well. In
particular, it is immediate from the proofs of the isomorphisms
(see~\cite{LOTHomPair}) that the isomorphism $\eta$ takes a generator
$\x\subset \alphas\cap\betas\subset \Sigma$ to the same subset of
$\Sigma$.

The second ingredient in the definition of $\CFI$ is relating
$\conj{\HD}$ and $\HD$ by a sequence of Heegaard moves. In the
bordered setting this is not possible: $\conj{\HD}$ is
$\beta$-bordered while $\HD$ is $\alpha$-bordered. The
Auroux-Zarev piece $\AZ$ comes to the rescue. Specifically, if we glue
$\AZ$ (respectively $\bAZ$) to $\conj{\HD}$ along the $\beta$-boundary of
$\AZ$ or $\bAZ$ then we have
\[
  \conj{\HD}\cup_\bdy\AZ\sim \HD\sim \conj{\HD}\cup_\bdy\bAZ
\]
\cite[Lemma 4.6]{LOTHomPair} (where $\sim$ means the diagrams are related by a sequence of bordered Heegaard moves or, equivalently, represent the same bordered $3$-manifold).

Now, fix bordered Heegaard diagrams $\HD_0,\HD_1$ with
$\bdy\HD_0=\PMC=-\bdy\HD_1$. Let $Y=Y(\HD_0\cup_\bdy\HD_1)$ be the
closed $3$-manifold represented by $\HD_0\cup_\bdy\HD_1$.  We show in
Theorem~\ref{thm:iota-right}
that, up to homotopy, the involution $\iota$ on $\CFa(Y)$ is the
composition of the following maps:
\begin{equation}\label{eq:decomp-iota}
\begin{split}
  \CFa(Y)&\simeq \CFAa(\HD_0)_{\Alg(\PMC)}\DT\lsup{\Alg(\PMC)}\CFDa(\HD_1)\\
  &\stackrel{\eta}{\longrightarrow}\CFAa(\conj{\HD_0})_{\Alg(\PMC)}\DT\lsup{\Alg(\PMC)}\CFDa(\conj{\HD_1})\\
  &=\CFAa(\conj{\HD_0})_{\Alg(\PMC)}\DT\lsup{\Alg(\PMC)}[\Id_{\Alg(\PMC)}]_{\Alg(\PMC)}\DT\lsup{\Alg(\PMC)}\CFDa(\conj{\HD_1})\\
   &\stackrel{\Omega_1}{\longrightarrow} \CFAa(\conj{\HD_0})_{\Alg(\PMC)}\DT\lsup{\Alg(\PMC)}\CFDAa(\bId_{\PMC})_{\Alg(\PMC)}\DT\lsup{\Alg(\PMC)}\CFDa(\conj{\HD_1})\\
  &\stackrel{\Omega_2}{\longrightarrow} 
    \CFAa(\conj{\HD_0})_{\Alg(\PMC)}\DT\lsup{\Alg(\PMC)}\CFDAa(\bAZ)_{\Alg(\PMC)}\DT\lsup{\Alg(\PMC)}\CFDAa(\AZ)_{\Alg(\PMC)}\DT \lsup{\Alg(\PMC)}\CFDa(\conj{\HD_1})\\
  &\stackrel{\Psi=\Psi_0\DT\Psi_1}{\longrightarrow}\CFAa(\HD_0)_{\Alg(\PMC)}\DT\lsup{\Alg(\PMC)}\CFDa(\HD_1)\\
  &\simeq \CFa(Y).
\end{split}
\end{equation}
Here, $\lsup{\Alg(\PMC)}[\Id_{\Alg(\PMC)}]_{\Alg(\PMC)}$ is the (type \DA) \emph{identity bimodule} of $\Alg(\PMC)$, i.e., the identity for the operation $\DT$~\cite[Definition 2.2.48]{LOT2}, while $\bId_{\PMC}$ is the standard bordered
Heegaard diagram for the identity map of $F(\PMC)$.
The map $\Omega_1$ is induced by a homotopy equivalence between $[\Id_{\Alg(\PMC)}]$ and $\CFDAa(\bId_{\PMC})$, while $\Omega_2$ is
induced by a sequence of Heegaard moves from $\bId_{\PMC}$ to the bordered
Heegaard diagram $\bAZ\cup\AZ$. The map $\Psi_0$ is induced by a
sequence of Heegaard moves from $\conj{\HD_0}\cup \bAZ$ to $\HD_0$ and the map $\Psi_1$ is induced by a sequence of Heegaard moves
from $\AZ\cup\conj{\HD_1}$ to $\HD_1$.

To give an algorithm to compute $\HFIa(Y)$ we restrict to the case that the $\HD_i$ come from a Heegaard splitting of $Y$. As discussed above, we can compute $\CFAa(\HD_0)$ and $\CFDa(\HD_1)$ in this case. Further, the diagrams $\AZ$ and $\bAZ$ are nice (both in the technical and colloquial sense) and so it is routine to compute $\CFDAa(\AZ)$ and $\CFDAa(\bAZ)$. We write down these bimodules explicitly in Section~\ref{sec:AZ}. To compute $\HFIa(Y)$ it remains to compute the maps $\Omega=\Omega_2\circ\Omega_1$ and $\Psi=\Psi_0\DT\Psi_1$. It turns out that both are determined by being the unique graded homotopy equivalences of the desired form; this is explained in Section~\ref{sec:rigid} (Lemmas~\ref{lem:HB-rigid} and~\ref{lem:MCG-rigid}). In particular, one never needs to compute $\CFDAa(\Id_{\PMC})$. (These rigidity results were first observed in unpublished work of Ozsv\'ath, Thurston, and the second author, and parallel results in Khovanov homology \cite{Khovanov06:cobordism}.)

An arguably even nicer description of $\iota$, in terms of morphisms
complexes, is given in Section~\ref{sec:hom-pair}.

Changing topics slightly, given a closed $3$-manifold $Y$, the based
mapping class group of $Y$ acts on
$\HFa(Y)$~\cite{OS06:HolDiskFour,JT:Naturality}. One can use bordered
Floer homology to compute the mapping class group action in a similar
way to $\HFIa$, so we explain that algorithm here as
well. (We are interested in this action partly because it sometimes
allows one to compute the concordance invariant
$q_\tau$~\cite{HLS:HEquivariant}.)

So, fix a closed $3$-manifold $Y$, a basepoint $p\in Y$, and a mapping
class $[\chi]\in\MCG(Y,p)$. We can choose a Heegaard splitting
$Y=\HB_0\cup_F\HB_1$ for $Y$ and a representative $\chi$ for
$[\chi]$ so that $\chi$ respects the Heegaard splitting, i.e.,
$\chi(\HB_i)=\HB_i$ (Lemma~\ref{lem:preserve-HS}). Let $\psi$ denote
the gluing map for the Heegaard splitting, so $\CFa(Y)$ is computed by
Equation~\eqref{eq:compute-CFa}, and we know how to compute
$\CFAa(H_g,\phi_0)$ and $\CFDa(H_g,\phi_0\circ\psi)$. Let $\chi|_F$
denote the restriction of $\chi$ to $F$. As described above, we
can also compute $\CFDAa(\chi|_F)$. Since $\chi|_F$ extends over $\HB_i$,
the bordered manifolds $(H_g,\phi_0)$ and
$(H_g,\phi_0\circ\chi|_F^{-1})$ are equivalent, as are the bordered
manifolds $(H_g,\phi_0\circ\psi)$ and
$(H_g,\phi_0\circ\psi\circ\chi|_F^{-1})$. Thus, there are
(grading-preserving) chain homotopy equivalences
\begin{align*}
  \CFAa(H_g,\phi_0)\DT \CFDAa(\chi|_F)&\stackrel{\Theta_0}{\longrightarrow}\CFAa(H_g,\phi_0)\\
  \CFDAa(\chi|_F^{-1})\DT \CFDa(H_g,\phi_0\circ\psi)&\stackrel{\Theta_1}{\longrightarrow}\CFDa(H_g,\phi_0\circ\psi).
\end{align*}
In fact, we show in Section~\ref{sec:rigid} that there are unique
graded homotopy equivalences $\Theta_0$ and $\Theta_1$ between these
modules (up to homotopy), so $\Theta_0$ and $\Theta_1$ are
algorithmically computable (cf.\ Section~\ref{sec:comp-htpy-equiv}).
We show in Theorem~\ref{thm:MCG-act-is}
that the action of $\chi$ on $\HFa(Y)$ is given by the composition
\begin{equation}\label{eq:MCG-act-is}
\begin{split}
  \CFa(Y)&\simeq \CFAa(H_g,\phi_0)\DT\CFDa(H_g,\phi_0\circ\psi)\\
  &=\CFAa(H_g,\phi_0)\DT[\Id_{\Alg(\PMC)}]\DT\CFDa(H_g,\phi_0\circ\psi)\\
  &\longrightarrow
    \CFAa(H_g,\phi_0)\DT\CFDAa(\chi|_F)\DT\CFDAa(\chi|_F^{-1})\DT\CFDa(H_g,\phi_0\circ\psi)\\
  &\stackrel{\Theta_0\DT\Theta_1}{\longrightarrow}
    \CFAa(H_g,\phi_0)\DT\CFDa(H_g,\phi_0\circ\psi)\\
  &\simeq\CFa(Y)
\end{split}
\end{equation}
for an appropriate homotopy equivalence
$[\Id_{\Alg(\PMC)}]\to \CFDAa(\chi|_F)\DT\CFDAa(\chi|_F^{-1})$. Again, there is a
unique such homotopy equivalence, so this map is computable.

The paper has two more contents. In Section~\ref{sec:ibF} we give a definition of \emph{involutive bordered Floer homology}, which describes succinctly what information one needs to compute about a bordered $3$-manifold in order to recover $\HFIa$ of gluings. In Section~\ref{sec:triangle} we use this description to prove a surgery exact triangle for involutive Heegaard Floer homology. (Previously, Lin proved that $\Pin(2)$-equivariant monopole Floer homology admits a surgery exact triangle~\cite[Theorem 1]{Lin:Exact}, but surgery triangles for involutive Heegaard Floer homology have so far been elusive.) 

This paper is organized as follows. In Section~\ref{sec:background} we
collect the results we need from the bordered Floer literature. Section~\ref{sec:comp-htpy-equiv} notes that, given two explicit, finitely generated type $D$, $A$, or \DA\ bimodules over the bordered algebras, computing the set of homotopy equivalences between them can be done algorithmically. The rigidity results---that there is a unique isomorphism between type $D$ or $A$ modules for the same bordered handlebody, and between type \DD, \DA, or \AAm\ modules for the same mapping cylinder---are proved in Section~\ref{sec:rigid}. The fact that Formula~\eqref{eq:decomp-iota} computes the map $\iota$ is proved in Section~\ref{sec:ibF}, which also proposes a general definition of involutive bordered Floer homology. Section~\ref{sec:MCG} shows that Formula~\eqref{eq:MCG-act-is} computes the mapping class group action on $\HFa$. The proof of the surgery triangle is in Section~\ref{sec:triangle}. Another computation of $\iota$, entirely in terms of type $D$ modules, is given in Section~\ref{sec:hom-pair}. We conclude with computer computations for the branched double covers of 10-crossing knots, in Section~\ref{sec:examples}.

\emph{Acknowledgments.} We thank Nick Addington, Tony Licata, Tye
Lidman, Ciprian Ma\-no\-les\-cu, Peter Ozsv\'ath, and Dylan Thurston
for helpful conversations. In particular, the results in
Section~\ref{sec:rigid} were first recorded in an unpublished paper of
Ozsv\'ath, Thurston, and the second author. Finally, we thank the referee for helpful suggestions.

\section{Background}\label{sec:background}
We assume the reader has a passing familiarity with bordered Heegaard Floer homology. The review in this section is focused on fixing notation and recalling some of the less well-known aspects of the theory such as gradings and the Auroux-Zarev diagram.

\subsection{The split pointed matched circle and its algebra}
Let $\PMC_k$ denote the \emph{split pointed matched circle} for a surface of genus $k$. That is, $\PMC_k=(Z,\{a_1,\dots,a_{4k}\},M,z)$ where $M$ matches $a_{4i+1}\leftrightarrow a_{4i+3}$, and $a_{4i+2}\leftrightarrow a_{4i+4}$, for $i=0,\dots,k-1$. Note that the matched pairs in $\PMC_k$ are in canonical bijection with $\{1,\dots,2k\}$, by identifying $\{a_{4i+1},a_{4i+3}\}\mapsto 2i+1$ and $\{a_{4i+2},a_{4i+4}\}\mapsto 2i+2$. 

The algebra $\Alg(\PMC_k)$ has a canonical $\Field$-basis of
\emph{strand diagrams}, and decomposes as a direct sum
\[
\Alg(\PMC_k)=\bigoplus_{i=-k}^k\Alg(\PMC,i).
\]
The integer $i$ denotes the \emph{weight} or \emph{$\SpinC$-structure} of a
strand diagram, which is the number of non-horizontal strands plus
half the number of horizontal strands minus $k$~\cite[Definition 3.23]{LOT1}. Only the summand $\Alg(\PMC_k,0)$ will be relevant in this paper, and we will often abuse notation and let $\Alg(\PMC_k)$ denote $\Alg(\PMC_k,0)$.

It will be convenient to have names for certain elements of
$\Alg(\PMC_k)$. Given a subset $\mathbf{s}\subset \{1,\dots,2k\}$ with
cardinality $k$ there is a corresponding basic idempotent
$I(\mathbf{s})\in \Alg(\PMC_k,0)$. Next, for $1\leq i<j\leq 4k$ let
$\rho_{i,j}$ be the chord from $a_i$ to $a_j$. There is a
corresponding algebra element $a(\rho_{i,j})\in\Alg(\PMC_k,0)$, the
sum of all strand diagrams obtained by adding $2k-2$ horizontal
strands to $\rho_{i,j}$ in any allowed way. To keep notation simple, we will often denote $a(\rho_{i,j})$ by $\rho_{i,j}$.

In the special case that $k=1$, $\Alg(\PMC_1,0)$ has $8$ generators: $I(1)$, $I(2)$, $\rho_{1,2}$, $\rho_{2,3}$, $\rho_{3,4}$, $\rho_{1,3}$, $\rho_{2,4}$, and $\rho_{1,4}$. The multiplication satisfies, for instance, $\rho_{1,2}\rho_{2,3}=\rho_{1,3}$ and $I(1)\rho_{1,2}I(2)=\rho_{1,2}$.

Note that $\PMC_k$ is symmetric under reflection: $-\PMC_k\cong \PMC_k$.

There is an inclusion map 
\[
\iota\co \overbrace{\Alg(\PMC_1)\otimes\cdots\otimes\Alg(\PMC_1)}^k\hookrightarrow \Alg(\PMC_k)
\]
which sends $\rho_{i,j}$ in the $\ell\th$ copy of $\Alg(\PMC_1)$ to $\rho_{4(\ell-1)+i,4(\ell-1)+j}$. There is also a projection map 
\[
\pi\co \Alg(\PMC_k)\to \overbrace{\Alg(\PMC_1)\otimes\cdots\otimes\Alg(\PMC_1)}^k
\]
satisfying $\pi\circ\iota=\Id_{\Alg(\PMC_1)^{\otimes k}}$ and $\pi(\rho)=0$ if $\rho$ is a strand diagram not in the image of $\iota$. (These are special cases of the maps in~\cite[Section 3.4]{LOT2}.)

\subsection{Explicit descriptions of some bordered handlebodies}\label{sec:comp-hb}

Let $Y_0$ be the \emph{$0$-framed solid torus}. The type $D$ structure $\CFDa(Y_0)$ has a single generator $n$ with 
\[
  \delta^1(n)=\rho_{1,3}n.
\]
The $\Ainf$-module $\CFAa(Y_0)$ also has a description with a single generator, but more convenient for us will be the model with three generators $t,u,v$,
\begin{align*}
  m_1(u)&=v &
  m_2(u,\rho_{1,2})&=t\\
  m_2(u,\rho_{1,3})&=v &
  m_2(t,\rho_{2,3})&=v,
\end{align*}
and all other $\Ainf$ operations vanish. In particular, this model for $\CFAa(Y_0)$ is an ordinary \dg module.
(The conventions are chosen so that $\CFAa(Y_0)\DT_{\Alg(\PMC_1)}\CFDa(Y_0)\cong\HFa(S^2\times S^1)\cong \Field\oplus\Field$.)

More generally, let $Y_{0^k}$ be the \emph{$0$-framed handlebody of
  genus $k$}. Then the \emph{standard type $D$ structure for
  $Y_{0^k}$}, denoted $\CFDa(Y_{0^k})$, is the image of
$\CFDa(Y_0)^{\otimes k}$ under the induction map
$\lsup{\Alg(-\PMC_1)^{\otimes
    k}}\ModCat\to\lsup{\Alg(-\PMC_k)}\ModCat$ associated to
$\iota$~\cite[Definition 2.2.48]{LOT2}. Equivalently, if $\lsup{\Alg(-\PMC_k)}[\iota]_{\Alg(-\PMC_1)^{\otimes k}}$ denotes the rank 1 \DA\ bimodule associated to $\iota$ then 
\[
  \CFDa(Y_{0^k})=\lsup{\Alg(-\PMC_k)}[\iota]_{\Alg(-\PMC_1)^{\otimes k}}\DT\left(\lsup{\Alg(-\PMC_1)}\CFDa(Y_0)\right)^{\otimes k}.
\]
The module $\CFAa(Y_{0^k})$ is the image of $\CFAa(Y_0)^{\otimes k}$ under the restriction map $\ModCat_{\Alg(\PMC_1)^{\otimes k}}\to\ModCat_{\Alg(\PMC_k)}$ associated to $\pi$. Equivalently,
\[
  \CFAa(Y_{0^k})_{\Alg(\PMC_k)}=\left(\CFAa(Y_0)_{\Alg(\PMC_1)}\right)^{\otimes k}\DT \lsup{\Alg(\PMC_1)^{\otimes k}}[\pi]_{\Alg(\PMC_k)}
\]

Explicitly, the type $D$ structure $\CFDa(Y_{0^k})$ has a single generator $n$ with
\[
\delta^1(n)=(\rho_{1,3}+\rho_{5,7}+\cdots+\rho_{4k-3,4k-1})n.
\]
The module $\CFAa(Y_{0^k})$ has basis $\{t,u,v\}^k$. The module structure is determined as follows. First, operations $m_i$, $i>2$, vanish: $\CFAa(Y_{0^k})$ is an honest \dg module. Second, $m_2(\cdot,\rho_{4i,4i+1})$ and $m_2(\cdot,\rho_{4i+3,4i+4})$ vanish identically. Third, given a basis element $(x_1,\dots,x_{k})\in \CFAa(Y_{0^k})$, 
\begin{align*}
  m_1((x_1,\dots,x_k))&=\sum_{x_i=u}(x_1,\dots,x_{i-1},v,x_{i+1},\dots,x_k)\\
  m_2((x_1,\dots,x_k),\rho_{4i+1,4i+2})&=\begin{cases}
	(x_1,\dots,x_{i},t,x_{i+2},\dots,x_k) & x_{i+1}=u\\
	0 & \text{otherwise}
  \end{cases}\\
  m_2((x_1,\dots,x_k),\rho_{4i+2,4i+3})&=\begin{cases}
	(x_1,\dots,x_{i},v,x_{i+2},\dots,x_k) & x_{i+1}=t\\
	0 & \text{otherwise}
  \end{cases}\\
  m_2((x_1,\dots,x_k),\rho_{4i+1,4i+3})&=\begin{cases}
	(x_1,\dots,x_{i},v,x_{i+2},\dots,x_k) & x_{i+1}=u\\
	0 & \text{otherwise.}
  \end{cases}
\end{align*}

\subsection{The type \texorpdfstring{\DD}{DD} identity bimodule}\label{sec:DD-id}
Fix a pointed matched circle $\PMC$ with orientation-reverse $-\PMC$.
Let $Y_{\Id}$ be the identity cobordism of $F(\PMC)$. Given a subset $\mathbf{s}\subset \{1,\dots,2k\}$, let $\mathbf{s}^c$ denote the complement of $\mathbf{s}$. Then $\CFDDa(\bId_\PMC)\coloneqq \CFDDa(Y_{\Id})$ is generated by $\{(I(\mathbf{s})\otimes I(\mathbf{s}^c))\}\subset \Alg(\PMC)\otimes\Alg(-\PMC)$. The differential is defined by 
\[
\delta^1(I(\mathbf{s})\otimes I(\mathbf{s}^c))=\sum_{\mathbf{t}\subset \{1,\dots,2k\}}\sum_{\rho\in\Chord(\PMC)}(I(\mathbf{s})\otimes I(\mathbf{s}^c))(a(\rho)\otimes a(-\rho))\otimes(I(\mathbf{t})\otimes I(\mathbf{t}^c))
\]
where $\Chord(\PMC)$ denotes the set of chords in the pointed matched
circle $\PMC$, and $-\rho$ is the chord in the orientation-reverse
$-\PMC$ associated to $\rho$.

\subsection{The Auroux-Zarev piece}\label{sec:AZ}

The \emph{Auroux-Zarev interpolating piece}~\cite{Auroux10:Bordered,Zarev:JoinGlue}, is the $\alpha$-$\beta$-bordered Heegaard diagram $\AZ(\PMC)$ defined as follows. For fixed $k$, let $T$ be the triangle defined by the $y$-axis, the $x$-axis, and the line $x+y=4k+1$. Let $e_y$ be the edge of $T$ along the $y$-axis, $e_x$ be the edge along the $x$-axis, and $e_D$ be the diagonal edge. Produce a genus $k$ surface $\Sigma'$ from $T$ by identifying small neighborhoods of the points $(i, 4k+1-i)$ and $(j, 4k+1-j)$ on $e_D$ whenever $i$ and $j$ are matched in $\PMC$. If $i$ and $j$ are matched in $\PMC$, the two vertical segments $T \cap \{x=i\}$ and $T \cap \{x=j\}$ descend to a single arc; declare this to be a $\beta$-arc. Similarly, the two horizontal segments $T \cap \{y=4k+1-i\}$ and $T \cap \{4k+1-j\}$ descend to a single arc; declare this to be an $\alpha$-arc. Finally, attach a one-handle connecting small neighborhoods of $(0,0)$ and $(4k+1,0)$, giving a surface $\Sigma$. Place the basepoint $z$ at $(0,4k+1)$. Then $\AZ(\PMC)=(\Sigma,\alphas,\betas,z)$, and the boundary of $\AZ(\PMC)$ is $\PMC\cup\PMC^\beta$. See Figure~\ref{fig:AZ}.

\begin{figure}
  \centering
  \includegraphics[width=\textwidth]{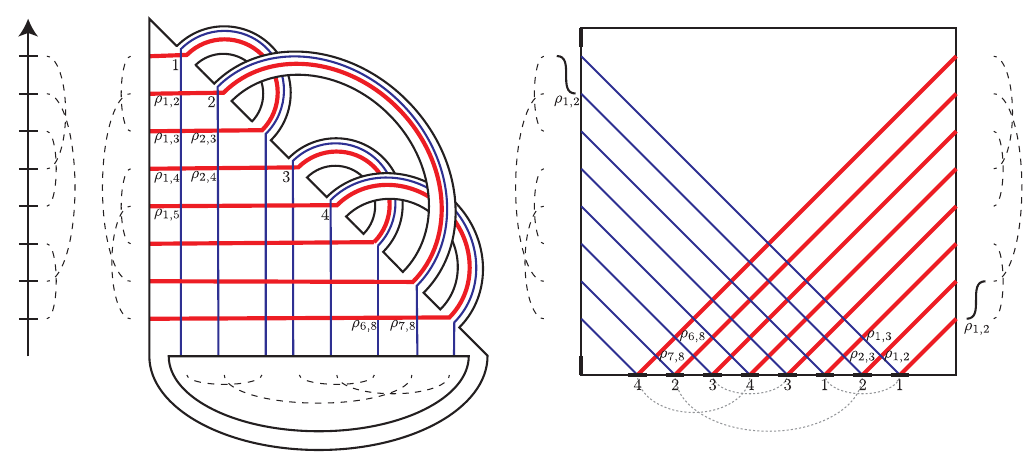}
  \caption{\textbf{The diagram $\AZ$.} Left: a pointed matched circle
    $\PMC$. Center: the diagram $\AZ(\PMC)$, with \textcolor{red}{$\alpha$-arcs thick} and \textcolor{blue}{$\beta$-arcs thin}. Some of the
    intersection points are labeled by the corresponding $\alpha$-arc if
    they correspond to a pair of horizontal strands, or the chord $\rho_{i,j}$ otherwise. Labels of generators are to the lower-left of the corresponding intersection point. Right: the same diagram, drawn to show the $\Alg(\PMC)$-actions on the left and right; the chord $\rho_{1,2}$ in each algebra is indicated. The thick segments are identified in pairs to give an orientable surface of genus $2$ with two boundary components; along the bottom, this identification is indicated by the dotted arcs.}
  \label{fig:AZ}
\end{figure}

There is a canonical identification between the set of generators
$\Gens(\AZ(\PMC))$ and the strand diagram basis for $\Alg(\PMC)$ as follows \cite{Auroux10:Bordered, LOTHomPair}. Numbering the $\alpha$-arcs from the top and the $\beta$-arcs from the left, the number of points in $\alpha_s \cap \beta_t$ is two if $s=t$ and otherwise is equal to the number of chords in $\PMC$ starting at an endpoint of $\alpha_t$ and ending at an endpoint of $\alpha_s$. If the endpoints of $\beta_s$ are $(i,0)$ and $(j,0)$, the intersection point in $\alpha_s \cap \beta_s$ which lies on $e_D$ corresponds to the smeared horizontal strand $\{i,j\}$. Other intersection points correspond to upward-sloping chords as follows: if $z$ lies at coordinates $(x,4k+1-y)$, then $z$ corresponds to the strand $\rho_{x,y}$ in $\Alg(\PMC)$. Figure~\ref{fig:AZ} indicates the identifications between intersection points in $\AZ(\PMC)$ and chords in $\Alg(\PMC)$. An arbitrary element of $\Gens(\AZ(\PMC))$ is a set of such intersection points, and corresponds to a strand diagram in $\Alg(\PMC)$. 

Using the fact that $\AZ(\PMC)$ is nice, it is easy to see that the differential on $\CFAAa(\AZ(\PMC))$ corresponds to the differential on $\Alg(\PMC)$. Furthermore, $m_2$ multiplications correspond to $k$-tuples of half-strips on the appropriate boundary \cite[Proposition 8.4]{LOTHomPair}. If we treat the $\alpha$ boundary as the right action and the $\beta$ boundary as the left action, we have 
\[
\CFAAa(\lsup{\beta}\AZ(\PMC)^{\alpha}) \simeq \lsub{\Alg(\PMC)}\Alg(\PMC)_{\Alg(\PMC)}
\]
whereas if we treat the $\alpha$ boundary as the left action and the $\beta$ boundary as the right action, we have
\[
\CFAAa(\lsup{\alpha}\AZ(\PMC)^{\beta}) \simeq \lsub{\Alg(-\PMC)}\Alg(-\PMC)_{\Alg(-\PMC)}.
\]
In our computations in Section 4, we will use $\AZ(-\PMC)$ (for $\PMC$ the split pointed matched circle), and treat the $\alpha$-boundary as the left action and the $\beta$-boundary as the right action. Then,
\[
\CFAAa(\lsup{\alpha}\AZ(-\PMC)^{\beta}) \simeq \lsub{\Alg(\PMC)}\Alg(\PMC)_{\Alg(\PMC)}.
\]
The corresponding labeling of generators is shown in Figure~\ref{fig:AZ2}.

We are also interested in a related diagram $\bAZ(\PMC)$ obtained from
$\AZ(\PMC)$ by switching the $\alpha$ and $\beta$
curves. (Equivalently, one could reflect $\AZ(-\PMC)$ across the
$x$-axis, obtaining $\bAZ(\PMC)=-\AZ(-\PMC)$.) Let
$\overline{\Alg(\PMC)}$ be the dual, over $\Field$, of $\Alg(\PMC)$. Since $\Alg(\PMC)$ comes with a preferred basis, the strand diagrams, there is a preferred basis $\{a^*\mid a\text{ is a strand diagram for }\Alg(\PMC)\}$ for $\overline{\Alg(\PMC)}$. The differential $\bar{d}$ on $\overline{\Alg(\PMC)}$ is the transpose of the differential $d$ on $\Alg(\PMC)$. Moreover, $\overline{\Alg(\PMC)}$ has left and right multiplications by $\Alg(\PMC)$: on the right, $a_1^*\cdot a_2$ is the element of $\overline{\Alg(\PMC)}$ which sends an element $a_3$ to $a_1^*(a_2a_3)$, and on the left $a_2 \cdot a_1^*$ is the element of $\overline{\Alg(\PMC)}$ which sends an element $a_3$ to $a_1^*(a_3 a_2)$. 

By the same computation as above one obtains
\begin{equation}\label{eq:CFAA-AZ-bar}
\CFAAa(\lsup{\beta}\bAZ(\PMC)^{\alpha}) \simeq \lsub{\Alg(\PMC)}\overline{\Alg(\PMC)}_{\Alg(\PMC)}
\end{equation}
if the $\alpha$-action is on the right \cite[Appendix A]{LOTHomPair}. See also Figure~\ref{fig:AZ2}.

\begin{figure}
  \centering
  \includegraphics[width=\textwidth]{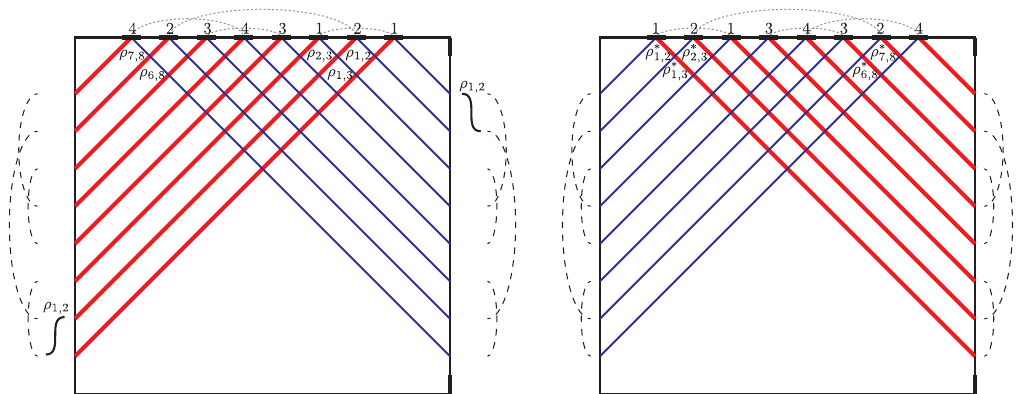}
  \caption{\textbf{The diagrams $\AZ(-\PMC)$ and $\bAZ(\PMC)$.} Left: the diagram $\AZ(-\PMC)$ for $\PMC$ the same pointed matched circle as in Figure~\ref{fig:AZ}, labeled compatibly with a left action by $\Alg(\PMC)$ corresponding to the $\alpha$-boundary and a right-action by $\Alg(\PMC)$ corresponding to the $\beta$-boundary. Right: the diagram $\bAZ(\PMC)$. Viewing the $\alpha$-boundary as the right action and the $\beta$-boundary as the left action, this is a bimodule over $\Alg(\PMC)$. Labels are to the left of the corresponding intersection point.}
  \label{fig:AZ2}
\end{figure}

Next we describe $\CFDAa(\lsup{\alpha}\AZ(-\PMC)^{\beta})$ in the case that the $\alpha$ boundary gives the left type $D$ structure and the $\beta$ boundary gives the right type $A$ structure. From the pairing theorem,
\begin{align*}
  \lsup{\Alg(\PMC)}\CFDAa(\lsup{\alpha}\AZ(-\PMC)^{\beta})_{\Alg(\PMC)}&\simeq \lsup{\Alg(\PMC)}\CFDDa(\bId_{-\PMC})^{\Alg(\PMC)}\DT \lsub{\Alg(\PMC)}\CFAAa(\lsup{\beta}\AZ(-\PMC)^{\alpha})_{\Alg(\PMC)}\\
  &=\lsup{\Alg(\PMC)}\CFDDa(\bId_{-\PMC^\beta})^{\Alg(\PMC)}\DT\lsub{\Alg(\PMC)}\Alg(\PMC)_{\Alg(\PMC)}.
\end{align*}
Thus, a generator of $\CFDAa(\lsup{\beta}\AZ(-\PMC)^{\alpha})$ corresponds to $J \otimes a$, where $a$ is a strand diagram in $\Alg(\PMC)$ and $J$ is the complementary idempotent to the left idempotent $I$ of $a$.  The map $\delta^1_2 \co \CFDAa(\AZ(-\PMC))\otimes \Alg(\PMC) \to \Alg(\PMC) \otimes \CFDAa(\AZ(-\PMC))$ is given by multiplication on the right; the image of $\delta^1_2$ is contained in the subspace $\CFDAa(\AZ(-\PMC))=1\otimes\CFDAa(\AZ(-\PMC))$.
The map $\delta_1^1 \co \CFDAa(\AZ(-\PMC)) \to \Alg(\PMC) \otimes \CFDAa(\AZ(-\PMC))$ is given by

\begin{align*}
  \delta_1^1(J \otimes a) = J \otimes (J \otimes d(a)) + \hspace{-2em}\sum_{\substack{\rho \in \Chord(\PMC)\\(J',I')\text{ complementary}}}\hspace{-2em} J a(\rho)J' \otimes ( J'\otimes I'a(\rho) \cdot a)
\end{align*}
All higher operations $\delta^1_k$, $k\geq 3$, vanish.

The same argument, but using Equation~\eqref{eq:CFAA-AZ-bar}, leads to
the following description of
$\CFDAa(\lsup{\beta}\bAZ(\PMC)^{\alpha})$. As needed by our
application, we will treat the $\beta$ boundary as the left action and
the $\alpha$ boundary as the right action. Generators of
$\CFDAa(\lsup{\beta}\bAZ(\PMC)^{\alpha})$ correspond to $J \otimes
a^*$, where $a$ is a strand diagram in $\PMC$, $a^*$ is the
corresponding basis element of $\overline{\Alg(\PMC)}$ and $J$ is the
complementary idempotent to the left idempotent $I$ of $a^*$ (or,
equivalently, the right idempotent $I$ of $a$). The map $\delta^1_2$
is given by $\delta^1_2(J \otimes a_1^*, a_2) = J\otimes \bigl(J \otimes (a_1^*\cdot a_2)\bigr)$.  The map
$\delta_1^1 \co \CFDAa(\bAZ(\PMC)) \to \Alg(\PMC) \otimes \CFDAa(\bAZ(\PMC))$
is given by
\begin{align*}
\delta_1^1(J \otimes a^*) = J \otimes (J \otimes \bar{d}(a^*)) + \hspace{-2em}\sum_{\substack{\rho \in \Chord(\PMC)\\(J',I')\text{ complementary}}}\hspace{-2em} J a(\rho)J' \otimes ( J'\otimes I'a(\rho) \cdot a^*).
\end{align*}
All higher operations $\delta^1_k$, $k\geq 3$, vanish.

To conclude this section, we recall some gluing properties
of the diagrams $\AZ$ and $\bAZ$ from \cite{LOTHomPair}:

\begin{lemma} \cite[Corollary 4.5]{LOTHomPair} The Heegaard diagram $\lsup{\alpha}\AZ(-\PMC)^{\beta} \cup \lsup{\beta}\bAZ(\PMC)^{\alpha}$ represents the identity map of $F(\PMC^{\alpha})$, and the diagram $\lsup{\beta}\bAZ(-\PMC)^{\alpha} \cup \lsup{\alpha}\AZ(\PMC)^{\beta}$ represents the identity map of $F(\PMC^{\beta})$.
\end{lemma}

\begin{lemma}\cite[Corollary 4.6]{LOTHomPair}
Let $\HD$ be an $\alpha$-bordered Heegaard diagram for $(Y, \phi \co F(\PMC) \to \bdy Y)$. Then the Heegaard diagram $\HD^{\beta}\cup \lsup{\beta}\AZ(-\PMC)^{\alpha}$ represents the three-manifold $(-Y, \phi \co F(-\PMC) \to -\bdy Y)$. In particular, $\HD^{\beta}\cup \lsup{\beta}\AZ(-\PMC)^{\alpha}$ and $-\HD$ represent the same bordered three-manifold.
\end{lemma}

\begin{convention}
  In the rest of the paper, we will typically drop $\PMC$ from the Auroux-Zarev piece, writing $\AZ$ (respectively $\bAZ$) to denote $\AZ(\PMC)$ or $\AZ(-\PMC)$ (respectively $\bAZ(\PMC)$ or $\bAZ(-\PMC)$) as appropriate. Whether $\PMC$ or $-\PMC$ is required is determined by the boundary of the diagram.
\end{convention}

\subsection{Gradings on bordered Floer modules}\label{sec:gradings}
A key step in our computations is knowing that there are unique graded homotopy equivalences between certain modules and bimodules (as formulated in Section~\ref{sec:rigid}). Here we review enough of the gradings in bordered Floer homology to make this statement precise. More details can be found in the original papers~\cite[Chapter 10]{LOT1},~\cite[Sections 2.5, 3.2, 6.5]{LOT2}.

Fix a pointed matched circle $\PMC$ representing a surface $F(\PMC)$. The algebra $\Alg(\PMC)$ is graded by a group $G(\PMC)$ which is a central extension
\[
  \ZZ\to G(\PMC)\to H_1(F(\PMC)).
\]
%
Let $\lambda$ be a generator for the central $\ZZ$. For homogeneous
elements $a,b\in\Alg(\PMC)$, differential satisfies $\gr(\bdy(a))=\lambda^{-1}\gr(a)$, and the multiplication satisfies $\gr(ab)=\gr(a)\gr(b)$.

Given a bordered $3$-manifold $Y$ with boundary parameterized by
$F(\PMC)$, $\CFAa(Y)$ is graded by a right $G(\PMC)$-set $S_A(Y)$, and
$\CFDa(Y)$ is graded by a left $G(-\PMC)$-set $S_D(Y)$. The $G$-orbits
in these sets correspond to the $\SpinC$-structures on $Y$. Similarly, if $Y$ is a cobordism from $F(\PMC_1)$ to $F(\PMC_2)$ then $\CFDAa(Y)$ is graded by a set $S_{\mathit{DA}}(Y)$ with a left action by $G(-\PMC_1)$ and a right action by $G(\PMC_2)$; $\CFDDa(Y)$ is graded by a set $S_{\mathit{DD}}(Y)$ equipped with commuting left actions by $G(-\PMC_1)$ and $G(-\PMC_2)$; and $\CFAAa(Y)$ is graded by a set $S_{\mathit{AA}}(Y)$ equipped with commuting right actions by $G(\PMC_1)$ and $G(\PMC_2)$.
The group $G(-\PMC)$ is the opposite group to $G(\PMC)$, so a left $G(-\PMC)$-set is the same data as a right $G(\PMC)$-set; $S_A(Y)$ and $S_D(Y)$ are related in this way. (Of course, all groups are isomorphic to their opposites, but here it is convenient to maintain the distinction.)

The $G(\PMC)$-grading on the bordered (bi)modules depends on a choice
of \emph{grading refinement data}~\cite[Section 10.5]{LOT1}. However,
up to homotopy equivalence, the bordered invariants are independent of
this choice~\cite[Proposition 6.32]{LOT2}.

The special cases of interest to us are:
\begin{enumerate}
\item Handlebodies. Suppose $Y$ is a handlebody of genus $g$. Then there is a unique $\SpinC$-structure on $Y$. The corresponding $G(\PMC)$-set $S_D(Y)$ is the quotient of $G(-\PMC)$ by a subgroup isomorphic to $\ZZ^g$, which projects isomorphically to $\ker[H_1(F(\PMC))\to H_1(Y)]\subset H_1(F(\PMC))$. In particular, the grading element $\lambda$ acts freely on $S_D(Y)$. 
\item Mapping cylinders of diffeomorphisms. If $\phi\co F(\PMC_1)\to F(\PMC_2)$ is a strongly based diffeomorphism and $Y_\phi$ is the associated arced cobordism then $S_{\mathit{DA}}(Y)$ is a free, transitive $G(-\PMC_1)$-set, and also a free, transitive $G(\PMC_2)$-set. Similar statements hold for $S_{\mathit{DD}}(Y)$ and $S_{\mathit{AA}}(Y)$.
\end{enumerate}

Given type $D$ structures $\lsup{\Alg(-\PMC)}P$ and $\lsup{\Alg(-\PMC)}Q$, graded by $G(-\PMC)$-sets $S$ and $T$, respectively, the chain complex of type $D$ structure morphisms $\Mor^{\Alg(-\PMC)}(P,Q)$ inherits a grading by the $\ZZ$-set $S^*\times_{G(-\PMC)} T$~\cite[Section 2.5.3]{LOT2}, where $S^*$ is the right $G(-\PMC)$-set with elements $s^*$ in bijection with $S$ and action $s^*\cdot g=(g^{-1} \cdot s)^*$~\cite[Definition 2.5.19]{LOT2}. The $\ZZ$-action persists because $\lambda$ is central in $G(-\PMC)$. The situation for $\Ainf$-modules and the various types of bimodules is similar. A morphism is \emph{homogeneous} if it lies in a single grading.

So, if $G(-\PMC)$ acts transitively on the grading set $S$ for $\lsup{\Alg(-\PMC)}P$ then the
complex $\Mor^{\Alg(-\PMC)}(P,P)$ is graded by $S^*\times_{G(-\PMC)}S\cong (S\times S)/G$ as $\ZZ$-sets. A morphism has \emph{grading 0} if it lands in the summand corresponding to $(s,s)\in (S\times S)/G$ for some (or equivalently, any) $s\in S$.

\begin{example}\label{eg:CFD-gr} Let $Y$ be a $0$-framed solid torus,
  and consider $\CFDa(Y)$. Since $\delta^1(n)=\rho_{1,3}n$, the
  gradings satisfy $\gr(\rho_{1,3}x)=\lambda^{-1}\gr(x)$.  Thus, the
  homomorphism $\CFDa(Y)\to\CFDa(Y)$, $x\mapsto \rho_{1,3}x$
  has degree $\lambda^{-1}\neq 0$.
\end{example}



\section{Computation of homotopy equivalences}\label{sec:comp-htpy-equiv}
Two key steps in our descriptions of involutive Floer homology and the
mapping class group involve computing homotopy equivalences between
$\Ainf$-modules or between type \DA\ bimodules.  
We explain in this section that the bordered algebras have finiteness
properties which imply that these computations can be carried out to
any order desired.

\begin{lemma}\label{lem:short-fact}
  Given a pointed matched circle $\PMC$ there is an integer $K$ so
  that any product of $n>K$ chords in $\Alg(\PMC)$ vanishes.
\end{lemma}
\begin{proof}
  This is immediate from the fact that no two strands in a strand
  diagram can start at the same point in the matched
  circle. So, if $\PMC$ represents a surface of genus $k$, 
  \[
    K=1+2+\cdots+4k-1=2k(4k-1)
  \]
  suffices. (This bound is not optimal.)
\end{proof}

\begin{proposition}\label{prop:compute-morphism}
  Fix a \dg algebra $\Blg$ and let $M$ and $N$ be type \DA\ bimodules
  over $\Blg$ and $\Alg(\PMC)$ where $\PMC$ is a pointed
  matched circle. Let $K$ be as in Lemma~\ref{lem:short-fact}. Suppose $\ell\geq K$ and
  $\{f^1_{1+n}\co M\otimes \Alg(\PMC)^{\otimes n}\to \Blg\otimes N\}_{n=0}^{\ell}$
  satisfy the type \DA\ homomorphism relations with up to $\ell+1$ inputs.
  Then there is a type \DA\
  module homomorphism $g\co M\to N$ so that $g^1_{1+n}=f^1_{1+n}$ for
  all $0\leq n\leq \ell$.
\end{proposition}
Since $\Ainf$-modules are a special case of type \DA\ bimodules, this
proposition covers $\Ainf$-modules as well. Roughly, the proposition
says that, after building a homomorphism which takes up to $K$ inputs,
one never gets stuck in extending the homomorphism to take one more
input.
\begin{proof}[Proof of Proposition~\ref{prop:compute-morphism}]
  View $\CFDDa(\bId_\PMC)$ as a left-right type \DD\ structure over $\Alg(\PMC)$ and $\Alg(\PMC)$.
  The functor $\cdot\DT\CFDDa(\bId_\PMC)$ gives an equivalence of
  categories from the category of type \DA\ bimodules over $\Blg$ and
  $\Alg(\PMC)$ to the category of (left-right) type \DD\ bimodules over $\Blg$ and
  $\Alg(\PMC)$. This functor sends a morphism $f\in\Mor(M,N)$ to
  $f\DT\Id_{\CFDDa(\bId_\PMC)}$.  As we will see, the key point is that the form of the
  differential $\delta^1$ on $\CFDDa(\bId_\PMC)$ and
  Lemma~\ref{lem:short-fact} imply that the map
  $f\DT\Id_{\CFDDa(\bId_\PMC)}$ depends only on the terms $f^1_{1+n}$
  for $n\leq K$.

  Fix data $f=\{f^1_{1+n}\}_{n=0}^{\ell}$ as in the statement of
  the proposition. Temporarily declare $f^1_i=0$ for $i>\ell$, and
  form $f\DT\Id_{\CFDDa(\bId_\PMC)}$. It follows from
  Lemma~\ref{lem:short-fact} and the form of $\delta^1$ on
  $\CFDDa(\bId_\PMC)$ (see also Section~\ref{sec:DD-id}) that 
  $f\DT\Id_{\CFDDa(\bId_\PMC)}$ is a type \DD\
  structure homomorphism. Since $\cdot\DT\CFDDa(\bId_\PMC)$ is a
  homotopy equivalence of \dg categories, there is a type \DA\ structure
  homomorphism $g$ so that
  $g\DT \Id_{\CFDDa(\bId_\PMC)}$ is homotopic to $f\DT\Id_{\CFDDa(\bId_\PMC)}$. So,
  $(g-f)\DT\Id_{\CFDDa(\bId_\PMC)}$ is nullhomotopic, so $g-f$ is itself
  nullhomotopic. Let $h$ be a nullhomotopy of $g-f$, i.e.,
  $g-f=d(h)$. Write $h=h'+h''$ where $h'$ consists of the terms with
  $\leq \ell+1$ inputs and $h''$ consists of the terms with $>\ell+1$
  inputs. Let $\tilde{f}=f+d(h'')$. Then $\tilde{f}^1_{1+n}=f^1_{1+n}$
  for all $n\leq \ell$. Further,
  \[
    \tilde{f}=g+d(h')
  \]
  so $\tilde{f}$ is a type \DA\ structure homomorphism. This proves the result.
\end{proof}

Proposition~\ref{prop:compute-morphism} implies that if $M$ and $N$
are homotopy equivalent then one can compute a homotopy
equivalence. First one finds terms with up to $K+1$ inputs satisfying
the type \DA\ structure relations with up to $K+1$ inputs, and so that this map has an up-to-$(K+1)$-input homotopy inverse. This is a finite (albeit
huge) computation. Proposition~\ref{prop:compute-morphism} then
implies that one can extend any such solution to more inputs, by
solving the type \DA\ structure relation inductively; one never gets
stuck.

Maybe a final word is in order about the meaning of the word
\emph{compute}. We have finitely generated modules $M$ and $N$ with
only finitely many non-zero operations. A type \DA\ structure
homomorphism from $M$ to $N$ is a computer program (Turing machine) $f$ which takes as
input an integer $\ell$ and inputs $m\in M$ and
$a_1,\dots,a_\ell\in\Alg(\PMC)$ and gives as output an element of
$N$. Being able to compute $f$ means we can write a computer program
$\mathscr{F}$ which takes as inputs homotopy equivalent modules $M$
and $N$ and outputs a computer program $f$ representing a type \DA\
homotopy equivalence from $M$ to $N$.

\section{Rigidity results}\label{sec:rigid}
In this section we prove that, up to homotopy, there are unique homogeneous
homotopy equivalences between certain modules.  The results in this
section were originally observed by P.\ Ozsv\'ath, D.\ Thurston, and
the second author.

We will call a map (and, in particular, a homotopy equivalence) $f$
\emph{homogeneous} if $f$ is homogeneous with respect to the grading on
morphism spaces (cf.~Section~\ref{sec:gradings}).
\begin{lemma}\label{lem:0-HB-rigid}
  Let $Y_{0^k}$ be the $0$-framed handlebody of genus $k$ and
  $\CFDa(Y_{0^k})$ the standard type $D$ module for $Y_{0^k}$ (as in
  Section~\ref{sec:comp-hb}). Then there is a unique homogeneous homotopy
  equivalence $\CFDa(Y_{0^k})\to\CFDa(Y_{0^k})$.
\end{lemma}
\begin{proof}
  Let $f^1\co \CFDa(Y_{0^k})\to\CFDa(Y_{0^k})$ be a homogeneous homotopy
  equivalence. Write 
  \[
    f^1(n)=(a_1+\cdots+a_m)n
  \]
  where the $a_i$ are strand diagrams (basic elements of $\Alg(-\PMC_k)$). Let
  $\mathscr{I}\subset\Alg(-\PMC_k)$ denote the ideal spanned by strand diagrams
  not of the form $I(\mathbf{s})$ (i.e., in which at
  least one strand is not horizontal). Then, as algebras,
  \[
    \Alg(-\PMC_k)/\mathscr{I}\cong \bigoplus_{\mathbf{s}\subset\{1,\dots,2k\}}\Field.
  \]
  Let $\CFDa(Y_{0^k})/\mathscr{I}$ be the result of extending scalars
  from $\Alg(-\PMC_k)$ to $\Alg(-\PMC_k)/\mathscr{I}$. Then
  $\CFDa(Y_{0^k})/\mathscr{I}$ is isomorphic to $\Field$, with trivial
  differential. Since $f^1$ must induce a homotopy equivalence
  \[
    \CFDa(Y_{0^k})/\mathscr{I}\to \CFDa(Y_{0^k})/\mathscr{I},
  \]
  it follows that one of the $a_i$, say $a_1$, is the idempotent
  $I(\{1,3,5,7,\dots\})$. That is,
  \[
    f^1(n)=n+(a_2+\cdots+a_m)n
  \]
  where $a_2,\dots,a_m\in\mathscr{I}$.
  
  Next we claim that $a_2=\cdots=a_m=0$. Since both the left
  and right idempotents of $a_i$ must agree with the
  left idempotent $I_n$ of $n$, the $a_i$ are in the algebra generated by 
  \[
    \{\rho_{1,3}I_n,\rho_{5,7}I_n,\cdots\}.
  \]
  As in Example~\ref{eg:CFD-gr}, 
  \[
    \gr(\rho_{4i+1,4i+3}n)=\lambda^{-1}\gr(n).
  \]
  Since $f^1$ is homogeneous and $n$ appears in $f^1(n)$, so every term
  in $f^1(n)$ has the same grading as $n$, it follows that $a_2=\cdots=a_m=0$ and so $f^1(n)=n$.
\end{proof}

Suppose that $\HD$ is a Heegaard diagram for a bordered handlebody and $P$ is a
$G$-set graded type $D$ structure homotopy equivalent to $\CFDa(\HD)$. Then
\[
  H_*\Mor^{\Alg(-\PMC)}(\CFDa(\HD),P)\cong \HFa\bigl((S^1\times S^2)^{\# k}\bigr)
\]
\cite[Theorem 1]{LOTHomPair} is graded by a free $\ZZ$-set (cf.~Section~\ref{sec:gradings}). Further, the homology
lies over a single $\ZZ$-orbit in the grading set. This $\ZZ$-orbit inherits a
total order, by declaring that $a>b$ if $a=\lambda^nb$ for some $n\in\ZZ$.
Thus, it makes sense to talk about a nontrivial homomorphism (that is, a homomorphism whose
image in $H_*\Mor^{\Alg(-\PMC)}(\CFDa(\HD),P)$ is nontrivial) of maximal
grading.  The same discussion holds for type $A$ invariants.

\begin{lemma}\label{lem:HB-rigid} 
  Let $\HD$ be a Heegaard diagram for a bordered handlebody and $P$
  (respectively $M$) a $G$-set-graded type $D$ structure
  (respectively $\Ainf$-module) homogeneous homotopy equivalent to $\CFDa(\HD)$. Then up
  to chain homotopy there is a \emph{unique} homogeneous homotopy equivalence
  $\CFDa(\HD)\to P$ (respectively $\CFAa(\HD)\to M$). Further, this
  homotopy class is represented by any non-trivial homomorphism of maximal
  grading.
\end{lemma}
So, if $\HD$ and $\HD'$ represent the same bordered handlebody, to
find a homotopy equivalence $\CFDa(\HD)\to\CFDa(\HD')$, say, it
suffices to find any grading-preserving, non-nullhomotopic
homomorphism.
\begin{proof}
  First, if $P$ and $Q$ are homotopy equivalent then the set of homotopy classes of homotopy equivalences from $P$ to $Q$ is a torseur for the set of homotopy classes of homotopy equivalences from $P$ to $P$. So, it suffices to prove the lemma in the case that $P=\CFDa(\HD)$ and $M=\CFAa(\HD)$.
  
  If $\HD_0$ represents the standard $0$-framed handlebody then by
  Lemma~\ref{lem:0-HB-rigid} there is a unique homogeneous homotopy equivalence
  $\CFDa(\HD_0)\to\CFDa(\HD_0)$.
  Next, there is a
  mapping class $\phi$ so that $\HD$ represents a handlebody with
  boundary parameterized by $\phi$. Then the pairing theorem gives a homogeneous homotopy
  equivalence
  \begin{equation}
    \label{eq:twist}
    \CFDAa(\phi)\DT\CFDa(\HD_0)\simeq \CFDa(\HD).
  \end{equation}
  Tensoring with $\CFDAa(\phi)$ is an equivalence of homotopy categories
  of $G$-set-graded type $D$ structures, with inverse
  $\CFDAa(\phi^{-1})\DT\cdot$~\cite[Corollary 8.1]{LOT2}, so the set of homotopy classes of
  homogeneous homotopy auto-equivalences of $\CFDAa(\phi)\DT\CFDa(\HD_0)$ is in
  bijection with the set of homotopy classes of homogeneous homotopy
  auto-equivalences of $\CFDa(\HD_0)$. Thus, by Equation~\eqref{eq:twist}
  there is a unique homotopy class of homogeneous homotopy auto-equivalences of
  $\CFDa(\HD)$. Finally,
  \[
    \CFAa(\HD)\simeq \CFAAa(\bId)\DT\CFDa(\HD).
  \]
  Since tensoring with $\CFAAa(\bId)$ is an equivalence of homotopy
  categories, with inverse given by tensoring with
  $\CFDDa(\bId)$~\cite[Corollary 8.1]{LOT2}, there is a unique homotopy class of
  homogeneous homotopy auto-equivalences of $\CFAa(\HD)$.

  For the second part of the statement, observe that any other
  non-trivial homogeneous
  homomorphism $\CFDa(\HD_0)\to\CFDa(\HD_0)$ has grading strictly
  smaller than the identity map. This property, too, is preserved by
  homotopy equivalences and equivalences of the homotopy category.
\end{proof}

There is an analogous result for the bimodules associated to mapping
classes:

\begin{lemma}\label{lem:DD-id-rigid}
  Let $\CFDDa(\bId)$ be the standard type \DD\ bimodule for the trivial
  cobordism (as in Section~\ref{sec:DD-id}). Then there is a unique
  homogeneous homotopy equivalence $\CFDDa(\bId)\to\CFDDa(\bId)$,
  which is also the unique nontrivial homomorphism of maximal grading.
\end{lemma}
\begin{proof}
  Since different choices of grading refinement data lead to graded
  chain homotopy equivalent modules $\CFDDa(\bId)$~\cite[Proposition
  6.32]{LOT2}, it suffices to prove the lemma for any choice of
  grading refinement data. Choose any grading refinement data for
  $\PMC$, and work with the induced grading refinement data for
  $-\PMC$. With respect to these choices, all of the generators of
  $\CFDDa(\bId)$ are in the same grading.

  Let $f^1\co\CFDDa(\bId)\to\CFDDa(\bId)$ be a homotopy equivalence.
  Write 
  \[
    f^1(I(\mathbf{s})\otimes
    I(\mathbf{s}^c))=\sum_{\mathbf{t}\subset\{1,\dots,2k\}}\sum_{i}(a_{\mathbf{s},\mathbf{t},i}\otimes
    a'_{\mathbf{s},\mathbf{t},i})\otimes (I(\mathbf{t})\otimes I(\mathbf{t}^c))
  \]
  where the $a_{\mathbf{s},\mathbf{t},i}$ and
  $a'_{\mathbf{s},\mathbf{t},i}$ are strand diagrams. Note that for
  each $\mathbf{s}$, $\mathbf{t}$, and $i$, 
  \begin{align*}
    I(\mathbf{s})a_{\mathbf{s},\mathbf{t},i}I(\mathbf{t})&=a_{\mathbf{s},\mathbf{t},i} &
    I(\mathbf{s}^c)a'_{\mathbf{s},\mathbf{t},i}I(\mathbf{t}^c)&=a'_{\mathbf{s},\mathbf{t},i}.
  \end{align*}

  Considering $\Alg(-\PMC)/\mathscr{I}$ and
  $\Alg(\PMC)/\mathscr{I}$ as in the proof of
  Lemma~\ref{lem:0-HB-rigid} shows that for each generator
  $I(\mathbf{s})\otimes I(\mathbf{s}^c)$, one of the terms
  $a_{\mathbf{s},\mathbf{s},i}$ must be
  $I(\mathbf{s})\otimes I(\mathbf{s}^c)$. We claim that these are the
  only terms in $f^1$.

  To see this note that the fact that $f^1$ is homogeneous implies
  that the supports of $a_{\mathbf{s},\mathbf{t},i}$ and
  $a'_{\mathbf{s},\mathbf{t},i}$ (in $H_1(Z,\mathbf{a})$) must be the
  same. (This statement depends on the fact that we are using
  corresponding grading refinement data for $\PMC$ and $-\PMC$.) That is,
  $a_{\mathbf{s},\mathbf{t},i}\otimes a'_{\mathbf{s},\mathbf{t},i}$
  lies in the \emph{diagonal subalgebra}~\cite[Definition
  3.1]{LOT4}. Every basic element in the diagonal subalgebra can be
  factored as a product of chord-like elements
  $a(\rho)\otimes a(-\rho)$~\cite[Lemma 3.5]{LOT4}. Since
  $(a(\rho)\otimes a(-\rho))\otimes(I(\mathbf{t})\otimes
  I(\mathbf{t}^c))$ occurs in the differential on $\CFDDa(\bId)$, it
  follows that the grading of a product of $n$ chord-like elements is
  $-n$. Thus, since $f^1$ is homogeneous, each term
  $a_{\mathbf{s},\mathbf{t},i}\otimes a'_{\mathbf{s},\mathbf{t},i}$
  must be a product of $0$ chord-like elements, i.e., have the form
  $I(\mathbf{s})\otimes I(\mathbf{s}^c)$. This proves the result.
\end{proof}

\begin{lemma}\label{lem:MCG-rigid}
  If $\phi\co F(\PMC)\to F(\PMC')$ is a mapping class and $M$ is a
  type \DA\ bimodule homogeneous homotopy equivalent to $\CFDAa(\phi)$
  (respectively $\CFAAa(\phi)$, $\CFDDa(\phi)$) then there is a unique
  homogeneous homotopy equivalence between $\CFDAa(\phi)$ (respectively
  $\CFAAa(\phi)$, $\CFDDa(\phi)$) and $M$. Further, the homotopy
  equivalence is the unique non-zero homotopy class of homomorphisms
  of maximal grading.
\end{lemma}
\begin{proof}
  Since tensoring with $\CFAAa(\bId)$ gives an equivalence of homotopy
  categories, it suffices to prove the statement for
  $\CFDDa(\phi)$. Further, since tensoring with $\CFDAa(\phi)$ gives
  an equivalence of categories, it suffices to prove the statement for
  $\CFDDa(\bId)$. Since the number of homotopy equivalences is
  preserved by homotopy equivalences, it suffices to show there is a
  unique homotopy equivalence $\CFDDa(\bId)\to\CFDDa(\bId)$ and that
  this homotopy equivalence is the unique non-nullhumotopic map of
  maximal grading. So, the result now follows from Lemma~\ref{lem:DD-id-rigid} and its proof.
\end{proof}

\begin{corollary}\label{cor:Phi}
  Up to homotopy, there is a unique homogeneous homotopy equivalence 
  \[
    \Omega\co \lsup{\Alg(\PMC)}[\Id_{\Alg(\PMC)}]_{\Alg(\PMC)}\stackrel{\simeq}{\longrightarrow} \lsup{\Alg(\PMC)}\CFDAa(\bAZ)_{\Alg(\PMC)}\DT\lsup{\Alg(\PMC)}\CFDAa(\AZ)_{\Alg(\PMC)}.
  \]
\end{corollary}
\begin{proof}
  Since $\bAZ\cup\AZ$ represents the identity diffeomorphism,
  this follows from the pairing theorem and Lemma~\ref{lem:MCG-rigid}.
\end{proof}

\section{Involutive bordered Floer homology}\label{sec:ibF}
We start by proving that the bordered description of $\CFIa$ in the
introduction does, in fact, give $\CFIa$:
\begin{theorem}\label{thm:iota-right}
  Fix bordered Heegaard diagrams $\HD_0,\HD_1$ with
  $\bdy\HD_0=\PMC=-\bdy\HD_1$. Let $Y=Y(\HD_0\cup_\bdy\HD_1)$ be the closed
  $3$-manifold represented by $\HD_0\cup_\bdy\HD_1$. Then, under the
  identification
  $\CFa(Y)\simeq \CFAa(\HD_0)_{\Alg(\PMC)}\DT\lsup{\Alg(\PMC)}\CFDa(\HD_1)$ from
  the pairing theorem~\cite[Theorem 3]{LOT1}, the map
  \[
    \Psi\circ\Omega\circ \eta\co\CFAa(\HD_0)_{\Alg(\PMC)}\DT\lsup{\Alg(\PMC)}\CFDa(\HD_1)\to \CFAa(\HD_0)_{\Alg(\PMC)}\DT\lsup{\Alg(\PMC)}\CFDa(\HD_1)
  \]
  from Formula~\eqref{eq:decomp-iota} is homotopic to the map
  $\iota\co\CFa(Y)\to\CFa(Y)$.
\end{theorem}
\begin{proof}
  In outline, the proof is that, up to homotopy, the map $\eta$ in Formula~\eqref{eq:decomp-iota} agrees with the map $\eta$ in the definition of $\HFIa$, while the composition $\Psi\circ\Omega$ agrees with the map $\Phi$ in the definition of $\HFIa$. To check this we need to verify that:
  \begin{enumerate}
  \item\label{item:easy-com} Up to homotopy, the following diagram commutes:
  \begin{equation}\label{eq:easy-com}
  \mathcenter{\xymatrix{
  \CFAa(\HD_0)\DT\CFDa(\HD_1)\ar[r]^\eta\ar[d] & \CFAa(\conj{\HD_0})\DT\CFDa(\conj{\HD_1})\ar[d]\\
  \CFa(\HD_0\cup\HD_1)\ar[r]_\eta & \CFa(\conj{\HD_0}\cup\conj{\HD_1}),
  }}
  \end{equation}
  where the vertical arrows come from the pairing theorem for bordered Floer homology.
\item\label{item:hard-com} Up to homotopy, the following diagrams commute, where in each case the bottom arrow is the chain homotopy equivalence on $\CFa$ (from~\cite{OS06:HolDiskFour,JT:Naturality}) induced by a sequence of Heegaard moves and the vertical arrows come from the pairing theorem:
  \begin{equation}\label{eq:com-Phi-1}
  \mathcenter{\xymatrix{
  \CFAa(\conj{\HD_0})\DT[\Id]\DT\CFDa(\conj{\HD_1})\ar[r]^-{\Omega_1}\ar[d] & \CFAa(\conj{\HD_0})\DT\CFDAa(\bId)\DT\CFDa(\conj{\HD_1})\ar[d]\\
  \CFa(\conj{\HD_0}\cup\conj{\HD_1})\ar[r] & \CFa(\conj{\HD_0}\cup\bId\cup\conj{\HD_1})
  }}
\end{equation}
  \begin{equation}\label{eq:com-Phi-2}
  \mathcenter{\xymatrix{
  \CFAa(\conj{\HD_0})\DT\CFDAa(\bId)\DT\CFDa(\conj{\HD_1})\ar[r]^-{\Omega_2}\ar[d] & {\begin{array}{l}\CFAa(\conj{\HD_0})\DT\CFDAa(\bAZ)\\
  \qquad\DT\CFDAa(\AZ)\DT\CFDa(\conj{\HD_1})\end{array}}\ar[d]\\
  \CFa(\conj{\HD_0}\cup\bId\cup\conj{\HD_1})\ar[r] & \CFa(\conj{\HD_0}\cup\bAZ\cup\AZ\cup\conj{\HD_1})
  }}
  \end{equation}
  and
  \begin{equation}\label{eq:com-Psi}
  \mathcenter{\xymatrix{
  \CFAa(\conj{\HD_0})\DT\CFDAa(\bAZ)\DT\CFDAa(\AZ)\DT\CFDa(\conj{\HD_1})\ar[d] \ar[r]^-{\Psi}& \CFAa(\HD_0)\DT\CFDa(\HD_1)\ar[d]\\
  \CFa(\conj{\HD_0}\cup\bAZ\cup\AZ\cup\conj{\HD_1}) \ar[r] & \CFa(\HD_0\cup\HD_1).
  }  }
  \end{equation}
\end{enumerate}
(Note that the top-left square of Diagram~\eqref{eq:com-Phi-1} is canonically isomorphic to $\CFAa(\conj{\HD_0})\DT\CFDa(\conj{\HD_1})$.)

The fact that Diagram~\eqref{eq:easy-com} commutes is straightforward from either proof of the pairing theorem. For example, the time-dilation proof~\cite[Chapter 9]{LOT1} has two steps. In the first, one chooses complex structures $j_n$ on $\HD_0\cup\HD_1$ with increasingly long necks around $\bdy\HD_0=\bdy\HD_1$. For $n$ sufficiently large, the differential on $\CFa(\HD_0\cup\HD_1)$ agrees with a count of pairs of holomorphic curves in $\HD_0$ and $\HD_1$, subject to a matching condition. We may as well assume that $\CFa(\HD_0\cup\HD_1)$ is computed with respect to one of these sufficiently large $j_n$. One then deforms the matching condition and observes that after a sufficiently large deformation the resulting differential agrees with $\CFAa(\HD_0)\DT\CFDa(\HD_1)$. Complexes with different deformation parameters are chain homotopy equivalent. Now, if one chooses the conjugate complex structure to $j_n$ on $\conj{\HD_0}\cup\conj{\HD_1}$ and then performs exactly the same deformation, at every stage the moduli spaces of holomorphic curves for $(\HD_0,\HD_1)$ and $(\conj{\HD_0},\conj{\HD_1})$ are identified. Thus, Diagram~\eqref{eq:easy-com} can be chosen to commute on the nose.
(The argument via the nice diagrams proof~\cite[Chapter 8]{LOT1} is even simpler, and is left as an exercise.)

  Consider next Diagram~\eqref{eq:com-Psi}. By a similar argument to the one just given, it suffices to show that the corresponding diagram
\[
  \xymatrix{
  \CFAa(\conj{\HD_0}\cup\bAZ)\DT\CFDa(\AZ\cup\conj{\HD_1})\ar[d] \ar[r]^{\Psi}& \CFAa(\HD_0)\DT\CFDa(\HD_1)\ar[d]\\
  \CFa(\conj{\HD_0}\cup\bAZ\cup\AZ\cup\conj{\HD_1}) \ar[r] & \CFa(\HD_0\cup\HD_1).
  } 
\]  
homotopy commutes. Recall that $\Psi$ is the box product of maps $\Psi_0$ and $\Psi_1$, induced by Heegaard moves from $\conj{\HD_0}\cup \bAZ$ to $\HD_0$ and
from $\AZ\cup\conj{\HD_1}$ to $\HD_1$, respectively. By definition, $\Psi_0\DT\Psi_1=(\Psi_0\DT\Id)\circ(\Id\DT\Psi_1)$, but this is canonically homotopic to $(\Id\DT\Psi_1)\circ(\Psi_0\DT\Id)$~\cite[Section 3.2]{LOT2}. Thus, we can break this into two steps, by considering the diagram
\[
  {\small
  \xymatrix{
  \CFAa(\conj{\HD_0}\cup\bAZ)\DT\CFDa(\AZ\cup\conj{\HD_1})\ar[d] \ar[r]^-{\Psi_0\DT\Id}&
  \CFAa(\HD_0)\DT\CFDa(\AZ\cup\conj{\HD_1}) \ar[r]^-{\Id\DT\Psi_1}\ar[d] & \CFAa(\HD_0)\DT\CFDa(\HD_1)\ar[d]\\
  \CFa(\conj{\HD_0}\cup\bAZ\cup\AZ\cup\conj{\HD_1}) \ar[r] & \CFa(\HD_0\cup\AZ\cup\conj{\HD_1})\ar[r] & \CFa(\HD_0\cup\HD_1).
  } }
\]
The proofs of commutativity of the two squares are essentially the same, so we will focus on the left square. 
We can relate $\conj{\HD_0}\cup\bAZ$ to $\HD_0$ by a sequence of bordered Heegaard moves; let $\HD^1,\HD^2,\cdots,\HD^k$ be the sequence of bordered Heegaard diagrams obtained by doing these moves one at a time, with $\HD^1=\conj{\HD_0}\cup\bAZ$ and $\HD^k=\HD_0$. There is a corresponding sequence of closed Heegaard diagrams 
\[
\HD^1\cup\AZ\cup\conj{\HD_1},\ \HD^2\cup\AZ\cup\conj{\HD_1},\ \cdots,\ \HD^k\cup\AZ\cup\conj{\HD_1},
\]
each successive pair of which is related by a Heegaard move. So, it suffices to check that:
\begin{lemma}\label{lem:inv-maps-agree}
If $\HD^i$ and $\HD^{i+1}$ are bordered Heegaard diagrams related by a bordered Heegaard move and $\HD'$ is another bordered Heegaard diagram with $\bdy\HD'=-\bdy\HD^i$ then the diagram
\[
  \xymatrix{
  \CFAa(\HD^i)\DT\CFDa(\HD')\ar[d] \ar[r]&
  \CFAa(\HD^{i+1})\DT\CFDa(\HD')\ar[d]\\
  \CFa(\HD^i\cup\HD') \ar[r] & \CFa(\HD^{i+1}\cup\HD').
  } 
\]
commutes up to homotopy. (Here, the horizontal arrows come from the invariance proofs for bordered and classical Heegaard Floer homology.)
\end{lemma}
\begin{proof}
For stabilizations (near the basepoint $z$), this is obvious: if $y$ is the intersection point between the new $\alpha$-circle and the new $\beta$-circle then both horizontal maps send a generator $\x$ to $\x\cup\{y\}$, and none of the moduli spaces used to define the vertical maps are affected. For handleslides, both horizontal maps are defined by counting holomorphic triangles, and the fact that this diagram commutes up to homotopy is a special case of the pairing theorem for triangles~\cite{LOT:DCov2}. For isotopies, commutativity follows by imitating the proof of the pairing theorem but with dynamic boundary conditions.
\end{proof}

Commutativity of Diagram~\eqref{eq:com-Phi-2} follows from a similar argument. Here, the horizontal maps come from a sequence of Heegaard moves relating the identity Heegaard diagram to the diagram $\bAZ\cup\AZ$. Working one Heegaard move at a time, the result follows from the obvious bimodule analogue of Lemma~\ref{lem:inv-maps-agree}.

For Diagram~\eqref{eq:com-Phi-1}, note that there are two homotopy equivalences
\[
  \CFAa(\conj{\HD_0})\cong \CFAa(\conj{\HD_0})\DT[\Id]\to \CFAa(\conj{\HD_0})\DT\CFDAa(\bId),
\]
one given by a sequence of Heegaard moves from $\conj{\HD_0}$ to $\conj{\HD_0}\cup\bId$ and the pairing theorem, and the other given by tensoring with the homotopy equivalence $[\Id]\simeq \CFDAa(\bId)$. The second of these is the map $\Omega_1$, while for the first of these Diagram~\eqref{eq:com-Phi-1} clearly commutes. So, it suffices to show these two maps are homotopic. In the case that $\HD_0$ represents a handlebody, this follows from Lemma~\ref{lem:HB-rigid}. For the general case, since tensoring with $\CFDDa(\Id)$ is a quasi-equivalence of \dg categories, it suffices to show that the two maps 
\[
  \CFAa(\conj{\HD_0})\DT\CFDDa(\bId)\to \CFAa(\conj{\HD_0})\DT\CFDAa(\bId)\DT\CFDDa(\bId),
\]
one induced by a sequence of Heegaard moves and the other induced by the equivalence $[\Id]\simeq \CFDAa(\bId)$, are homotopic.  By homotopy associativity of the box tensor product and the pairing theorem (see~\cite{LOT2}), it suffices to show that the two maps
\[
  \CFDDa(\bId)\to \CFDAa(\bId)\DT\CFDDa(\bId),
\]
one given by a sequence of Heegaard moves and the other by the
equivalence $[\Id]\simeq \CFDAa(\bId)$, are homotopic. This last statement follows from rigidity of $\CFDDa(\bId)$, Lemma~\ref{lem:DD-id-rigid}.
\end{proof}

Next we abstract the bordered information required to compute
involutive Heegaard Floer homology.

\begin{definition}
  Fix a pointed matched circle $\PMC$. An \emph{involutive type $D$
    module} over $\Alg(\PMC)$ consists of a pair
  $(\lsup{\Alg(\PMC)}P,\Psi_P)$ where $\lsup{\Alg(\PMC)}P$ is a type $D$
  structure over $\Alg(\PMC)$ and
  \[
    \Psi_P\co \lsup{\Alg(\PMC)}\CFDAa(\AZ)_{\Alg(\PMC)}\DT\lsup{\Alg(\PMC)}P\to \lsup{\Alg(\PMC)}P
  \]
  is a homotopy equivalence of type $D$ structures. We call two involutive
  type $D$ structures $(\lsup{\Alg(\PMC)}P,\Psi_P)$ and
  $(\lsup{\Alg(\PMC)}Q,\Psi_Q)$ \emph{equivalent} if there is a type $D$
  structure homotopy equivalence
  $g\co \lsup{\Alg(\PMC)}P\to\lsup{\Alg(\PMC)}Q$ so that $g\circ \Psi_P$
  is homotopic to $\Psi_Q\circ (\Id\DT g)$.
  
  Similarly, an \emph{involutive $\Ainf$-module} over $\Alg(\PMC)$ consists of a pair
  $(M_{\Alg(\PMC)},\Psi_M)$ where $M_{\Alg(\PMC)}$ is an $\Ainf$-module over $\Alg(\PMC)$ and
  \[
    \Psi_M\co M_{\Alg(\PMC)}\DT\lsup{\Alg(\PMC)}\CFDAa(\bAZ)_{\Alg(\PMC)}\to M_{\Alg(\PMC)}
  \]
  is a homotopy equivalence of $\Ainf$-modules. We call involutive
  $\Ainf$-modules $(M_{\Alg(\PMC)},\Psi_M)$ and
  $(N_{\Alg(\PMC)},\Psi_N)$ \emph{equivalent} if there is an
  $\Ainf$-module homotopy equivalence
  $g\co M_{\Alg(\PMC)}\to N_{\Alg(\PMC)}$ so that $g\circ\Psi_M$ is
  homotopic to $\Psi_N\circ (g\DT \Id)$.
\end{definition}

\begin{definition}
  Given a bordered $3$-manifold $Y$ with boundary $-F(\PMC)$ and
  bordered Heegaard diagram $\HD$ for $Y$, let
  $\CFDIa(\HD)=(\CFDa(\HD),\Psi_D)$ be the involutive type $D$ module
  where $\Psi_D$ is the map
  \[
    \lsup{\Alg(\PMC)}\CFDAa(\AZ)_{\Alg(\PMC)}\DT\lsup{\Alg(\PMC)}\CFDa(\HD)
    \stackrel{\simeq}{\longrightarrow} \lsup{\Alg(\PMC)}\CFDa(\AZ\cup \conj{\HD})\stackrel{\simeq}{\longrightarrow} 
    \lsup{\Alg(\PMC)}\CFDa(\HD)
  \]
  in which the first equivalence is given by the pairing theorem and the
  second is induced by some sequence of Heegaard moves from
  $\AZ\cup\conj{\HD}$ to $\HD$.

  Similarly, given a bordered $3$-manifold $Y$ with boundary $F(\PMC)$
  and bordered Heegaard diagram $\HD$ for $Y$, let
  $\CFAIa(\HD)=(\CFAa(\HD),\Psi_A)$ be the involutive $\Ainf$-module where
  $\Psi_A$ is the map
  \[
    \CFAa(\HD)_{\Alg(\PMC)}\DT\lsup{\Alg(\PMC)}\CFDAa(\bAZ)_{\Alg(\PMC)}
    \stackrel{\simeq}{\longrightarrow} \CFAa(\conj{\HD}\cup \bAZ)_{\Alg(\PMC)}\stackrel{\simeq}{\longrightarrow} 
    \CFAa(\HD)_{\Alg(\PMC)}
  \]
  in which the first equivalence is given by the pairing theorem and the
  second is induced by some sequence of Heegaard moves from
  $\conj{\HD}\cup\bAZ$ to $\HD$.
\end{definition}

\begin{conjecture}\label{conj:bord-inv}
  The involutive type $D$ structure $\CFDIa(\HD)$ and involutive
  $\Ainf$-module $\CFAIa(\HD)$ are invariants of the bordered
  $3$-manifold $Y$.
\end{conjecture}

The missing ingredient to prove Conjecture~\ref{conj:bord-inv} is an analogue of
Ozsv\'ath-Szab\'o-Juh\'asz-Thurston-Zemke's naturality theorem. That is, we do
not know that the maps $\Psi_A$ and $\Psi_D$ are independent of the choice of
sequence of Heegaard moves. In the special case that $Y$ is a handlebody,
Conjecture~\ref{conj:bord-inv} follows from
Lemma~\ref{lem:HB-rigid}. In general, it is
not even known that $\CFDIa(\HD)$ and $\CFAIa(\HD)$ are invariants of the
Heegaard diagram $\HD$, since as far as we know different sequences of Heegaard
moves would give different maps $\Psi_D$ and $\Psi_A$.

The rest of this paper does \emph{not} depend on Conjecture~\ref{conj:bord-inv}.

\begin{definition}
  The tensor product
  \[(M_{\Alg(\PMC)},\Psi_M)\DT(\lsup{\Alg(\PMC)}P,\Psi_P)
  \]
  of an involutive type $D$ structure $(\lsup{\Alg(\PMC)}P,\Psi_P)$ and an involutive $\Ainf$-module $(M_{\Alg(\PMC)},\Psi_M)$ is the mapping cone of the map
  \[
    \xymatrix{
      M\DT P\cong M\DT[\Id]\DT P \ar[rrrr]^-{\Id+[(\Psi_M\DT\Psi_P)\circ (\Id\DT \Omega\DT\Id)]} &&&& M\DT P
    }
  \]
  where $\Omega\co [\Id]\to \CFDAa(\bAZ)\DT\CFDAa(\AZ)$ is the
  homotopy equivalence from Corollary~\ref{cor:Phi}. This tensor
  product is a differential module over $\Field[Q]/(Q^2)$ in an
  obvious way.
\end{definition}

\begin{lemma}
  If $(\lsup{\Alg(\PMC)}P,\Psi_P)$ and $(\lsup{\Alg(\PMC)}Q,\Psi_Q)$ (respectively $(M_{\Alg(\PMC)},\Psi_M)$ and $(N_{\Alg(\PMC)},\Psi_N)$) are equivalent involutive type $D$ structures (respectively $\Ainf$-modules) over $\Alg(\PMC)$ then the box tensor products $(M_{\Alg(\PMC)},\Psi_M)\DT(\lsup{\Alg(\PMC)}P,\Psi_P)$ and $(N_{\Alg(\PMC)},\Psi_M)\DT(\lsup{\Alg(\PMC)}Q,\Psi_P)$ are quasi-isomorphic differential modules over $\Field[Q]/(Q^2)$.
\end{lemma}
\begin{proof}
  The proof is straightforward and is left to the reader.
\end{proof}

The following is the pairing theorem for involutive bordered Floer
homology:
\begin{theorem}
  Fix bordered Heegaard diagrams $\HD_1$ and $\HD_2$ with
  $\bdy\HD_1=\PMC=-\bdy\HD_2$. Then there is a chain homotopy
  equivalence
  \[
    \CFIa(\HD_1\cup_\bdy\HD_2)\simeq \CFAIa(\HD_1)\DT\CFDIa(\HD_2).
  \]
\end{theorem}
\begin{proof}
  This follows from Theorem~\ref{thm:iota-right}.
\end{proof}

\section{Computing the mapping class group action}\label{sec:MCG}
We start by recalling a well-known lemma:
\begin{lemma}\label{lem:preserve-HS}
  Let $\phi\co (Y,p)\to (Y,p)$ be an orientation-preserving, based
  diffeomorphism. Then there is a Heegaard splitting
  $Y=\HB_0\cup_\Sigma\HB_1$ with $p\in \Sigma$ and a diffeomorphism
  $\chi$ isotopic to $\phi$ (rel.~$p$) so that $\chi(\HB_i)=\HB_i$.
\end{lemma}
\begin{proof}
  Start with any Heegaard splitting $Y=\HB_0\cup_\Sigma\HB_1$ of
  $Y$. Then $\phi(\HB_0)\cup_{\phi(\Sigma)}\phi(\HB_1)$ is another
  Heegaard splitting of $Y$. Since any pair of Heegaard splittings
  becomes isotopic after sufficiently many stabilizations, after
  stabilizing enough times we may assume that $(\HB_0,\HB_1)$ is
  isotopic to $(\phi(\HB_0),\phi(\HB_1))$, by some ambient isotopy
  $\psi_t\co Y\to Y$. Consider the map $\psi_1^{-1}\circ\phi$. Since
  $\psi_1^{-1}$ is isotopic to the identity, $\psi_1^{-1}\circ\phi$ is
  isotopic to $\phi$. Clearly $\psi_1^{-1}\circ\phi$ preserves the
  Heegaard splitting $Y=\HB_0\cup_\Sigma\HB_1$.
\end{proof}

With notation as in the introduction, the main goal of this section is to prove:
\begin{theorem}\label{thm:MCG-act-is}
  The action of a mapping class $[\chi]$ on $\HFa(Y)$ is given by the
  composition of the maps in Formula~\eqref{eq:MCG-act-is}.
\end{theorem}
\begin{proof}
  The proof is similar to the proof of Theorem~\ref{thm:iota-right}.
  Choose a Heegaard splitting as in Lemma~\ref{lem:preserve-HS}. Let
  $F$ denote the Heegaard surface and $\psi\co F\to F$ the gluing
  diffeomorphism for the Heegaard splitting. Let
  $\HD_0=(\Sigma_0,\alphas^0,\betas^0,z^0)$ be a bordered Heegaard
  diagram representing the $0$-framed handlebody and
  $\HD_1=(\Sigma_1,\alphas^1,\betas^1,z^1)$ a bordered Heegaard
  diagram representing $(H_g,\phi_0\circ\psi)$, so
  $\HD_0\cup_\bdy\HD_1$ is a Heegaard diagram for $Y$. Here, we view
  $\Sigma_0$ and $\Sigma_1$ as subsets of $Y$; see Figure~\ref{fig:HeegAndBord}.

  \begin{figure}
    \centering
    \includegraphics[scale=.83333]{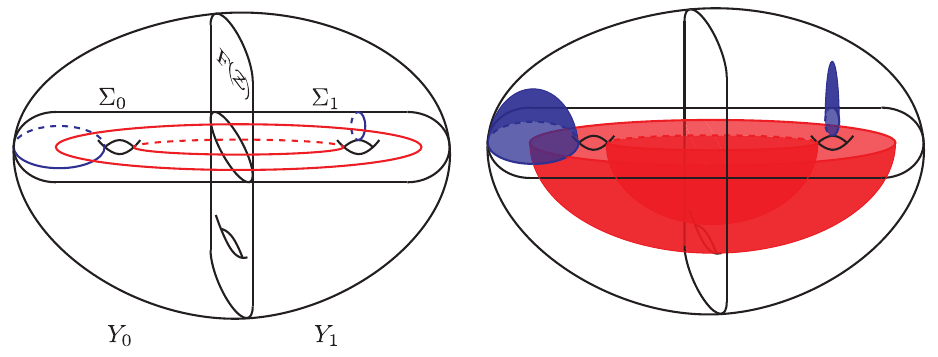}
    \caption{\textbf{Embedded bordered Heegaard surfaces.} Left: a schematic of how the bordered Heegaard surfaces $\Sigma_0$ and $\Sigma_1$ and the boundary $F(\PMC)=\bdy Y_0=-\bdy Y_1$ lie in $Y$. Right: a schematic of the descending disks of index $2$ critical points and ascending disks of index $1$ critical points.}
    \label{fig:HeegAndBord}
  \end{figure}
  
  Applying $\chi$ to $\Sigma_0$ and $\Sigma_1$ gives new Heegaard
  diagrams $(\chi(\Sigma_i),\chi(\alphas^i),\chi(\betas^i),\chi(z^i))$
  for $\HB_i$. (Abstractly, of course, these diagrams are
  diffeomorphic to the original ones, but they are new subsets of the
  manifolds $\HB_i$.) Let $C_\chi$ denote the mapping cylinder of
  $\chi|_F$, and let $\HD_\chi$ be a bordered Heegaard diagram for
  $C_\chi$. Cutting $Y$ along $F$ and gluing in $C_\chi C_{\chi^{-1}}$
  does not change the $3$-manifold. At the level of Heegaard diagrams,
  this corresponds to gluing $\HD_\chi$ to $\chi(\HB_0)$ and
  $\HD_{\chi^{-1}}$ to $\chi(\HB_1)$.
  Further, this cutting and regluing can be realized by a path of
  Heegaard diagrams from the standard Heegaard diagram for the identity map to $\HD_{\chi}\cup \HD_{\chi^{-1}}$.

  Now, $\chi(\HD_0)\cup\HD_\chi$ and $\HD_0$ are bordered Heegaard
  diagrams representing $\HB_0$, and the Heegaard surfaces are
  embedded so that they have the same boundary. Similarly,
  $\HD_{\chi^{-1}}\cup \chi(\HD_1)$ and $\HD_1$ both represent
  $\HB_1$. Choose a path of Heegaard diagrams from
  $\chi(\HD_0)\cup\HD_{\chi}$ to $\HD_0$, and a path from
  $\HD_{\chi^{-1}}\cup \chi(\HD_1)$ to $\HD_1$. By definition, the
  map on $\HFa$ induced by $\chi$ comes from the composition of the
  Heegaard Floer continuation map associated to the path which
  introduces $\HD_{\chi}\cup\HD_{\chi^{-1}}$ and then the Heegaard Floer
  continuation maps associated to the Heegaard moves from
  $\chi(\HD_0)\cup\HD_{\chi}$ to $\HD_0$ and
  $\HD_{\chi^{-1}}\cup \chi(\HD_1)$ to $\HD_1$.

  By the pairing theorem for holomorphic triangle
  maps~\cite[Proposition 5.35]{LOT:DCov2}, these continuation maps
  agree with the tensor products of the bordered continuation maps
  associated to the pieces which are changing. So, a similar argument
  to the proof of commutativity of
  Diagrams~(\ref{eq:com-Phi-1}),~(\ref{eq:com-Phi-2}),
  and~(\ref{eq:com-Psi}) shows that the action of
  $\chi$ on $\HFa$ is given by the composition
  \begin{align*}
    \CFAa(\HD_0)\DT\CFDa(\HD_1)&=\CFAa(\HD_0)\DT[\Id]\DT\CFDa(\HD_1)\\
    &\to\CFAa(\HD_0)\DT\CFDAa(\bId_{\PMC})\DT\CFDa(\HD_1)\\
    &\to \CFAa(\HD_0)\DT\CFDAa(\chi)\DT\CFDAa(\chi^{-1})\DT\CFDa(\HD_1)\\
    &\stackrel{\Theta_0\DT\Theta_1}{\longrightarrow} \CFAa(\HD_0)\DT\CFDa(\HD_1),
  \end{align*}
  where the first map comes from the homotopy equivalence
  $[\Id]\simeq \CFDAa(\bId_{\PMC})$, the second map comes from some homotopy
  equivalence $\CFDAa(\bId_{\PMC})\to \CFDAa(\chi|_F)\DT\CFDAa(\chi^{-1}|_F)$ and the
  third map comes from some homotopy equivalences
  $\CFAa(\HD_0)\DT\CFDAa(\chi|_F)\to \CFAa(\HD_0)$ and
  $\CFDAa(\chi^{-1}|_F)\DT\CFDa(\HD_1)\to \CFDa(\HD_1)$. By
  Lemmas~\ref{lem:HB-rigid} and~\ref{lem:MCG-rigid}, up to homotopy there is a
  unique homotopy equivalence in each case.
\end{proof}

As noted in the introduction, each of the maps in Formula~\eqref{eq:MCG-act-is} is the unique homotopy class of homotopy equivalences between the given source and target. So, after computing the modules and bimodules by factoring into mapping classes~\cite{LOT4}, computing the homotopy equivalences required to describe the mapping class group action is straightforward (and, in particular, algorithmic).

\section{The surgery exact triangle}\label{sec:triangle}

The goal of this section is to prove:
\begin{theorem}\label{thm:I-tri}
  Let $K$ be a framed knot in a $3$-manifold $Y$. Then there is a surgery exact triangle
  \[
    \cdots\to\HFIa(Y)\to\HFIa(Y_{-1}(K))\to \HFIa(Y_0(K))\to\HFIa(Y)\to\cdots.
  \]
\end{theorem}

Before turning to the proof, to fix notation we recall the modules and maps used in the bordered proof of the surgery exact triangle for $\HFa$~\cite[Section 11.2]{LOT1}. (The reader is referred to the original paper for a more leisurely account.)

Let $\HD_0$, $\HD_1$ and $\HD_\infty$ be the standard, genus $1$ Heegaard diagrams for the $0$-framed, $1$-framed, and $\infty$-framed solid tori, respectively. It is easy to compute that
\begin{align*}
  \CFDa(\HD_\infty)&=\langle r\mid \delta^1(r)=\rho_{2,4}r\rangle \\
  \CFDa(\HD_{-1})&=\langle a,b\mid \delta^1(a)=(\rho_{1,2}+\rho_{3,4})b,\ \delta^1(b)=0\rangle\\
  \CFDa(\HD_0)&=\langle n\mid \delta^1(n)=\rho_{1,3}n\rangle.
\end{align*}
Further, there is a short exact sequence
\[
  0\to \CFDa(\HD_\infty)\stackrel{\phi}{\longrightarrow}\CFDa(\HD_{-1})\stackrel{\psi}{\longrightarrow}\CFDa(\HD_0)\to 0
\]
where $\phi$ and $\psi$ are given by
\begin{align*}
  \phi(r)&=b+\rho_{2,3}a & \psi(a)&=n & \psi(b)&=\rho_{2,3} n.
\end{align*}
Given any bordered $3$-manifold $Y$ with boundary $T^2$, tensoring
this short exact sequence with $\CFAa(Y)$ gives a long exact sequence
in homology~\cite[Proposition 2.36]{LOT1}---the desired surgery exact
sequence. (This exact sequence agrees with Ozsv\'ath-Szab\'o's
original~\cite{OS04:HolDiskProperties}, as proved in~\cite[Corollary
5.41]{LOT:DCov2}.)


For notational convenience, in this section let $\AZ=\AZ(-\PMC_1)$.
%
The main work in extending these bordered computations to prove Theorem~\ref{thm:I-tri} is the
following lemma:
\begin{lemma}\label{lem:surj-loc-com}
  There are homotopies $G\co \CFDAa(\AZ)\DT\CFDa(\HD_\infty)\to \CFDa(\HD_{-1})$ and $H\co \CFDAa(\AZ)\DT\CFDa(\HD_{-1})\to \CFDa(\HD_0)$ making each square of the following diagram homotopy commute:
  \[
    \xymatrix{
      \CFDAa(\AZ)\DT\CFDa(\HD_\infty)\ar[r]^{\Id\DT\phi}\ar[d]_{\Psi}\ar@{-->}[dr]^G & \CFDAa(\AZ)\DT\CFDa(\HD_{-1})\ar[r]^{\Id\DT\psi}\ar[d]_{\Psi}\ar@{-->}[dr]^H & \CFDAa(\AZ)\DT\CFDa(\HD_0)\ar[d]_{\Psi}\\
      \CFDa(\HD_\infty)\ar[r]^{\phi} & \CFDa(\HD_{-1})\ar[r]^{\psi} & \CFDa(\HD_0).
    }
  \]
  Further, $\psi\circ G=H\circ(\Id\DT\phi)$.
\end{lemma}
\begin{proof}
  This is a direct computation. 

  Recall from Section~\ref{sec:AZ} that $\CFDAa(\AZ)$ is the type \DA\ bimodule with generators 
  \[
  \begin{matrix}
  \iota_1\otimes\iota_0, &\iota_1\otimes\rho_{1,2},&\iota_1\otimes\rho_{1,3},&\iota_1\otimes\rho_{1,4},\\\iota_1\otimes\rho_{3,4},&
  \iota_0\otimes\iota_1,&\iota_0\otimes\rho_{2,3},&\iota_0\otimes\rho_{2,4},
  \end{matrix}
  \]
  
  The operation $\delta^1_2\co \CFDAa(\AZ)\otimes\Alg(T^2)\to \Alg(T^2)\otimes\CFDAa(\AZ)$ 
  is the obvious right action of $\Alg(T^2)$, and $\delta^1_1\co \CFDAa(\AZ)\to \Alg(T^2)\otimes\CFDAa(\AZ)$ is induced by
  \begin{align*}
    \delta^1_1(\iota_1\otimes\iota_0)&=\rho_{2,3}\otimes(\iota_0\otimes\rho_{2,3}),\\
    \delta^1_1(\iota_0\otimes\iota_1)&=\rho_{1,2}\otimes (\iota_1\otimes\rho_{1,2})+\rho_{3,4}\otimes(\iota_1\otimes\rho_{3,4})+\rho_{1,4}\otimes(\iota_1\otimes\rho_{1,4}),
  \end{align*}
  and the Leibniz rule with $\delta^1_2$. All higher $\delta^1_k$, $k\geq 3$, vanish.

  Thus, the type $D$ structure $\CFDAa(\AZ)\DT\CFDa(\HD_\infty)$ has
  generators 
  \[
  \begin{matrix}
    \iota_1\otimes \rho_{1,2}\otimes r, & \iota_1\otimes\rho_{1,4}\otimes r, & \iota_1\otimes\rho_{3,4}\otimes r, &
    \iota_0\otimes\iota_1\otimes r, & \iota_0\otimes\rho_{2,4}\otimes r
  \end{matrix}
  \]
  (as a type $D$ structure) with differential given by
  \begin{align*}
    \delta^1(\iota_1\otimes \rho_{1,2}\otimes r) &=\greenit{\iota_1\otimes(\iota_1\otimes\rho_{1,4}\otimes r)}\\
    \delta^1(\iota_1\otimes\rho_{1,4}\otimes r) &=0\\
    \delta^1(\iota_1\otimes\rho_{3,4}\otimes r) &=\orangeit{\rho_{2,3}\otimes (\iota_0\otimes\rho_{2,4}\otimes r)} \\
    \delta^1(\iota_0\otimes\iota_1\otimes r) &=\greenit{\iota_0\otimes (\iota_0\otimes \rho_{2,4}\otimes r)}+\orangeit{\rho_{1,2}\otimes(\iota_1\otimes\rho_{1,2}\otimes r)}\\&\qquad\qquad
+\orangeit{\rho_{3,4}\otimes(\iota_1\otimes\rho_{3,4}\otimes r)}
    +\orangeit{\rho_{1,4}\otimes(\iota_1\otimes\rho_{1,4}\otimes r)}\\
    \delta^1(\iota_0\otimes\rho_{2,4}\otimes r)&=\orangeit{\rho_{1,2}\otimes(\iota_1\otimes\rho_{1,4}\otimes r)}.
  \end{align*}
  Here, \greenit{some terms} come from the operation $\delta^1$ on
  $\CFDa(\HD_\infty)$ (together with the operation $\delta^1_2$ on
  $\CFDAa(\AZ)$) while \orangeit{other terms} come from the operation
  $\delta^1_1$ on $\CFDAa(\AZ)$. The quasi-isomorphism $\Psi$ is given by
  \[
  \Psi(\iota_1\otimes\rho_{3,4}\otimes r) =\iota_1\otimes r\qquad
  \Psi(\iota_0\otimes\rho_{2,4}\otimes r) =\rho_{3,4}\otimes r,
  \]
  \[
  \Psi(\iota_1\otimes \rho_{1,2}\otimes r)=
  \Psi(\iota_1\otimes\rho_{1,4}\otimes r)=
  \Psi(\iota_0\otimes\iota_1\otimes r)=0.
  \]
  
  These formulas are perhaps easier to absorb, and check, graphically:
  \[
  \begin{tikzpicture}
    \node at (0,-5) (CFDA) {$\CFDAa(\AZ)\DT\CFDa(\HD_\infty)$};
    \node at (-1,0) (iota1) {$\iota_0|\iota_1|r$};
    \node at (-4,-4) (rho12) {$\iota_1|\rho_{1,2}|r$};
    \node at (-2,-3) (rho14) {$\iota_1|\rho_{1,4}|r$};
    \node at (0,-2) (rho24) {$\iota_0|\rho_{2,4}|r$};
    \node at (2,-1) (rho34) {$\iota_1|\rho_{3,4}|r$};
    \node at (8,-2) (r) {$r$};
    \node at (8.75,-2) (rphant) {};
    \node at (8,-5) (CFD) {$\CFDa(\HD_\infty)$};
    \draw[->, bend left=30] (r) to node[right]{\lab{\rho_{2,4}}} (rphant) to (r);
    \draw[->] (rho12) to (rho14);
    \draw[->] (rho34) to node[above,sloped]{\lab{\rho_{2,3}}} (rho24);
    \draw[->] (iota1) to (rho24);
    \draw[->] (iota1) to node[above,sloped]{\lab{\rho_{1,2}}} (rho12);
    \draw[->] (iota1) to node[above,sloped]{\lab{\rho_{3,4}}} (rho34);
    \draw[->] (iota1) to node[above, sloped]{\lab{\rho_{1,4}}} (rho14);
    \draw[->] (rho24) to node[above,sloped]{\lab{\rho_{1,2}}} (rho14);
    \draw[dmar] (rho34) to (r);
    \draw[dmar] (rho24) to node[above,sloped]{\lab{\rho_{3,4}}} (r);
  \end{tikzpicture}
  \]
  Here, we have replaced tensor signs with vertical bars. Unlabeled
  arrows are implicitly labeled by idempotents. Dashed arrows
  represent the map $\Psi$, while solid arrows represent
  $\delta^1$. Labels are always above the corresponding arrows. The
  check that $\Psi$ is a homomorphism reduces to examining all
  length-two paths from a vertex on the left to $r$. The map is
  clearly a quasi-isomorphism.

  After this warm-up, the complexes
  $\CFDAa(\AZ)\DT\CFDa(\HD_{-1})$ and $\CFDAa(\AZ)\DT\CFDa(\HD_0)$;
  the maps $\Psi$ on them; the morphisms $\phi$ and $\psi$ and induced
  maps $\Id\DT\phi$ and $\Id\DT\psi$; and the homotopies are shown in Figure~\ref{fig:tri-proof}.
  \begin{figure}
  \[
  \begin{tikzpicture}
    \node at (-1,0) (iota1r) {$\iota_1|r$};
    \node at (-4,-4) (rho12r) {$\rho_{1,2}|r$};
    \node at (-2,-3) (rho14r) {$\rho_{1,4}|r$};
    \node at (0,-2) (rho24r) {$\rho_{2,4}|r$};
    \node at (2,-1) (rho34r) {$\rho_{3,4}|r$};
    \node at (6,-2) (r) {$r$};
    \node at (6.75,-2) (rphant) {};
    \draw[->, bend left=60] (r) to node[right]{\lab{\rho_{2,4}}} (rphant) to (r);
    \draw[->] (rho12r) to (rho14r);
    \draw[->] (rho34r) to node[above,sloped]{\lab{\rho_{2,3}}} (rho24r);
    \draw[->] (iota1r) to (rho24r);
    \draw[->] (iota1r) to node[above,sloped]{\lab{\rho_{1,2}}} (rho12r);
    \draw[->] (iota1r) to node[above,sloped]{\lab{\rho_{3,4}}} (rho34r);
    \draw[->] (iota1r) to node[above, sloped]{\lab{\rho_{1,4}}} (rho14r);
    \draw[->] (rho24r) to node[above,sloped]{\lab{\rho_{1,2}}} (rho14r);
    \draw[dmar] (rho34r) to (r);
    \draw[dmar] (rho24r) to node[above,sloped]{\lab{\rho_{3,4}}} (r);
    \node at (-4,-9) (iota0a) {$\iota_0|a$};
    \node at (-2,-10) (rho23a) {$\rho_{2,3}|a$};
    \node at (0,-11) (rho13a) {$\rho_{1,3}|a$};
    \node at (0,-8) (iota1b) {$\iota_1|b$};
    \node at (2,-9) (rho34b) {$\rho_{3,4}|b$};
    \node at (4,-10) (rho24b) {$\rho_{2,4}|b$};
    \node at (3,-11) (rho14b) {$\rho_{1,4}|b$};
    \node at (-2,-7) (rho12b) {$\rho_{1,2}|b$};
    \draw[->] (iota0a) to (rho34b);
    \draw[->] (iota0a) to (rho12b);
    \draw[->] (rho23a) to (rho24b);
    \draw[->] (rho13a) to (rho14b);
    \draw[->] (iota1b) to node[above,sloped]{\lab{\rho_{1,2}}} (rho12b);
    \draw[->] (iota1b) to node[above,sloped]{\lab{\rho_{3,4}}} (rho34b);
    \draw[->] (rho34b) to node[above,sloped]{\lab{\rho_{2,3}}} (rho24b);
    \draw[->] (rho24b) to node[above,sloped]{\lab{\rho_{1,2}}} (rho14b);
    \draw[->] (iota0a) to node[above,sloped]{\lab{\rho_{2,3}}} (rho23a);
    \draw[->] (rho23a) to node[above,sloped]{\lab{\rho_{1,2}}} (rho13a);
    \draw[->, bend right=15] (iota1b) to node[above,sloped]{\lab{\rho_{1,4}}} (rho14b);
    \node at (8,-8) (a) {$a$};
    \node at (8,-10) (b) {$b$};
    \draw[->] (a) to node[above,sloped]{\lab{\rho_{1,2}+\rho_{3,4}}} (b);
    \draw[dmar] (iota1b) to (a);
    \draw[dmar] (rho34b) to (b);
    \draw[dmar] (rho12b) to (b);
    \node at (-2,-13) (iota0n) {$\iota_0|n$};
    \node at (1,-13) (rho23n) {$\rho_{2,3}|n$};
    \node at (-1,-15) (rho13n) {$\rho_{1,3}|n$};
    \node at (6,-14) (n) {$n$};
    \node at (6.75,-14) (nphant) {};
    \draw[->] (iota0n) to (rho13n);
    \draw[->] (iota0n) to node[above,sloped]{\lab{\rho_{2,3}}} (rho23n);
    \draw[->] (rho23n) to node[above,sloped]{\lab{\rho_{1,2}}} (rho13n);
    \draw[->, bend left=60] (n) to node[right]{\lab{\rho_{1,3}}} (nphant) to (n);
    \draw[dmar] (rho23n) to (n);
    \draw[dmar] (rho13n) to node[above,sloped]{\lab{\rho_{2,3}}} (n);
    %
    \draw[amar] (r) to (b);
    \draw[amar] (r) to node[above,sloped]{\lab{\rho_{2,3}}} (a);
    \draw[amar, bend left=75] (a) to (n);
    \draw[amar] (b) to node[above,sloped]{\lab{\rho_{2,3}}} (n);
    \draw[amar] (iota1r) to (iota1b);
    \draw[amar, bend left=10] (iota1r) to (rho23a);
    \draw[amar] (rho34r) to (rho34b);
    \draw[amar] (rho24r) to (rho24b);
    \draw[amar] (rho14r) to (rho14b);
    \draw[amar] (rho12r) to (rho12b);
    \draw[amar] (rho12r) to (rho13a);
    \draw[amar] (iota0a) to (iota0n);
    \draw[amar] (rho23a) to (rho23n);
    \draw[amar] (rho13a) to (rho13n);
    \draw[amar] (iota1b) to (rho23n);
    \draw[amar] (rho12b) to (rho13n);
    \draw[aar] (rho24r) to (a);
    \draw[aar] (rho14r) to (b);
    \draw[aar] (rho12r) to node[above,sloped]{\lab{\rho_{2,4}}} (b);
    \draw[aar] (rho24b) to (n);
    \draw[aar] (rho14b) to node[above,sloped]{\lab{\rho_{2,3}}} (n);
  \end{tikzpicture}
  \]
  \caption{\textbf{Proof of Lemma~\ref{lem:surj-loc-com}.} The maps $\Psi$ are dashed, $\phi$
    and $\psi$ are dotted, and the homotopies are thick. We have
    dropped the first idempotent in the label for each generator
    (since it is determined by the other data), so for instance the
    generator $\iota_1\otimes \rho_{3,4}\otimes r$ is denoted
    $\rho_{3,4}|r$. Arrow labels, which indicate type $D$ outputs, are always above the center of the corresponding arrow (except for the self-arrows of $n$ and $r$).}
    \label{fig:tri-proof}
  \end{figure}

  Again, checking that this diagram is correct reduces to looking at
  length-two paths. Have fun!
\end{proof}

\begin{proof}[Proof of Theorem~\ref{thm:I-tri}]
  The framing of $K$ makes $X(K)\coloneqq Y\setminus\nbd(K)$ into a
  bordered $3$-manifold. We claim that the squares in the following diagram commute
  up to the dashed homotopies shown:
  \[
    \xymatrix{
      \CFAa(X(K))\DT\CFDa(\HD_\infty)\ar[r]^{\Id\DT\phi}\ar[d]_-{\Omega} & \CFAa(X(K))\DT\CFDa(\HD_{-1})\ar[r]^{\Id\DT\psi}\ar[d]_-{\Omega} & \CFAa(X(K))\DT\CFDa(\HD_0)\ar[d]_-{\Omega}\\
      {\begin{matrix}\CFAa(X(K))\DT\CFDAa(\bAZ)\\\DT\CFDAa(\AZ)\DT\CFDa(\HD_\infty)\end{matrix}}\ar[r]^-{\Id^3\DT\phi}\ar[d]_-{\Psi_0\DT\Id^2} & {\begin{matrix}\CFAa(X(K))\DT\CFDAa(\bAZ)\\\DT\CFDAa(\AZ)\DT\CFDa(\HD_{-1})\end{matrix}}\ar[r]^-{\Id^3\DT\psi}\ar[d]_-{\Psi_0\DT\Id^2} & {\begin{matrix}\CFAa(X(K))\DT\CFDAa(\bAZ)\\\DT\CFDAa(\AZ)\DT\CFDa(\HD_0)\ar[d]_-{\Psi_0\DT\Id^2}\end{matrix}}\\
      {\begin{matrix}\CFAa(X(K))\DT\CFDAa(\AZ)\\\DT\CFDa(\HD_\infty)\end{matrix}}\ar[r]^-{\Id^2\DT\phi}\ar[d]_-{\Id\DT\Psi_1}\ar@{-->}[dr]^-{\Id\DT G} & {\begin{matrix}\CFAa(X(K))\DT\CFDAa(\AZ)\\\DT\CFDa(\HD_{-1})\end{matrix}}\ar[r]^-{\Id^2\DT\psi}\ar[d]_-{\Id\DT\Psi_1} \ar@{-->}[dr]^-{\Id\DT H}& {\begin{matrix}\CFAa(X(K))\DT\CFDAa(\AZ)\\\DT\CFDa(\HD_0)\ar[d]_-{\Id\DT\Psi_1}\end{matrix}}\\
      \CFAa(X(K))\DT\CFDa(\HD_\infty)\ar[r]^{\Id\DT\phi} & \CFAa(X(K))\DT\CFDa(\HD_{-1})\ar[r]^{\Id\DT\psi} & \CFAa(X(K))\DT\CFDa(\HD_0).
    }
  \]
  
  Indeed, the fact that the top two rows commute on the nose follows from basic properties of
  the box tensor product~\cite[Lemma 2.3.3]{LOT2}. For the third row, commutativity up to the homotopies follows from these
  properties and Lemma~\ref{lem:surj-loc-com}. Further, by Lemma~\ref{lem:surj-loc-com}, the homotopies satisfy
  \[
    (\Id\DT H)\circ (\Id^2\DT\phi)=(\Id\DT\psi)\circ (\Id\DT G).
  \]
  
  Since by Theorem~\ref{thm:iota-right} the
  composition of the three vertical arrows in any column is the map
  $\iota$, it follows that there is a homotopy commutative diagram
  \begin{equation}\label{eq:CF-tri-diag}
    \xymatrix{
      0 \ar[r] & \CFa(Y)\ar[r]^-i\ar[d]_-{\Id+\iota}\ar@{-->}[dr]^{G'} & \CFa(Y_{-1}(K))\ar[r]^-p\ar[d]_-{\Id+\iota}\ar@{-->}[dr]^{H'} & \CFa(Y_0(K))\ar[r]\ar[d]_-{\Id+\iota} & 0\\
      0 \ar[r] & \CFa(Y)\ar[r]^-{i} & \CFa(Y_{-1}(K))\ar[r]^-{p} & \CFa(Y_0(K))\ar[r] & 0
    }
  \end{equation}
  where the rows are short exact sequences inducing the surgery
  exact triangle on homology, and the diagonal arrows are the
  homotopies $G'=(\Id\DT G)\circ(\Psi_0\DT\Id^2)\circ \Omega$ and $H'=(\Id\DT H)\circ(\Psi_0\DT\Id^2)\circ\Omega$.

  The theorem now follows from the commutative diagram~\eqref{eq:CF-tri-diag} and homological algebra (cf.~\cite[Proof of Proposition 4.1]{HM:involutive}). That is,
  by
  Lemma~\ref{lem:surj-loc-com}, the homotopies in Diagram~\eqref{eq:CF-tri-diag} satisfy
  \begin{equation}
    \label{eq:homotopies-sat}
    p\circ G' = H'\circ i.
  \end{equation}
  Take the mapping cone of each vertical map in the diagram, to
  obtain a sequence of chain complexes
  \[
    0\to \Cone((\Id+\iota)_{\CFa(Y)})\to \Cone((\Id+\iota)_{\CFa(Y_{-1}(K))})\to\Cone((\Id+\iota)_{\CFa(Y_0(K))})\to 0
  \]
  where the maps are given by the matrices
  \[
    \begin{bmatrix}
      i & 0\\
      G & i
    \end{bmatrix}\qquad\text{and}\qquad
    \begin{bmatrix}
      p & 0\\
      H & p
    \end{bmatrix}.
  \]
  Homotopy commutativity of Diagram~\eqref{eq:CF-tri-diag}
  implies that these maps are chain maps, and
  exactness of the rows in Diagram~\eqref{eq:CF-tri-diag} together
  with Equation~\eqref{eq:homotopies-sat} implies that this sequence
  is exact. The associated long exact sequence is the statement of
  the lemma.
\end{proof}

\begin{remark}
  The proof of Theorem~\ref{thm:I-tri} also shows that the map induced by $\iota$ on homology commutes with the maps in the surgery exact triangle for $\HFa$. Lidman points out that this commutativity can be deduced more directly, by an argument that also applies to $\HF^\pm$. Specifically, the maps in the surgery exact triangle for $\HFa$ or $\HF^\pm$ are induced by cobordisms, and cobordism maps commute with the conjugation isomorphism (cf.~\cite[Theorem 3.6]{OS06:HolDiskFour}).
\end{remark}

\section{Involutive Floer homology as morphism spaces}\label{sec:hom-pair}
In this section we give some formulas purely in terms of $\CFDa$ for
the map $\iota\co\CFa(Y)\to\CFa(Y)$ and the map associated to a
mapping class, which may be helpful in computer implementations.

Given a type $D$ structure $\lsup{\Alg}P$ over a \dg algebra $\Alg$ over $\Field$,
consisting of a finite-dimensional underlying vector space $X$ and a
map $\delta^1\co X\to \Alg\otimes X$, the \emph{dual type $D$
  structure} $\overline{P}^{\Alg}$ has underlying vector space $X^*$, the dual space to $X$,
and operation
\[
  \delta^1_{\overline{P}}\co X^*\to X^*\otimes \Alg
\]
induced from $\delta^1\in\Hom(X,\Alg\otimes X)$ via the
identifications
\[
  \Hom(X,\Alg\otimes X)\cong X^*\otimes \Alg\otimes X\cong X\otimes X^*\otimes\Alg \cong \Hom(X^*,X^*\otimes\Alg).
\]

Given a bordered $3$-manifold $Y$ with boundary $F(\PMC)$, recall that
\[
  \CFAa(Y)\simeq \overline{\CFDa(-Y)}\DT\Alg(\PMC),
\]
\cite[Theorem 2]{LOTHomPair} so given bordered $3$-manifolds $Y_1$ and $Y_2$ with $\bdy Y_1=F(\PMC)=-\bdy Y_2$,
\begin{equation}\label{eq:hom-pairing}
  \begin{split}
    \CFa(Y_1\cup_\bdy Y_2)&\simeq\CFAa(Y_1)\DT\CFDa(Y_2)
    \simeq\overline{\CFDa(-Y_1)}\DT\Alg(\PMC)\DT\CFDa(Y_2)\\
    &=\Mor^{\Alg(\PMC)}(\CFDa(-Y_1),\CFDa(Y_2))
  \end{split}
\end{equation}
\cite[Theorem 1]{LOTHomPair}.

Using this, we explain how to compute the map $\iota$ without
mentioning $\CFAa$. Fix a Heegaard splitting $Y=\HB\cup_\psi\HB$. To
compute $\HFa(Y)$ one
first computes $\CFDa(\HB,\psi\circ \phi_0)$ and $\CFDa(\HB,\phi_0)$,
where $\phi_0\co F(\PMC_{g})\to\bdy\HB$ is the $0$-framing (as in
Section~\ref{sec:comp-hb}). The computation of
$\CFDa(\HB,\phi_0\circ \psi)$ uses a factorization of $\psi$ into arcslides and the identity
\[
  \CFDa(Y,\phi\circ\psi)\simeq \Mor^{\Alg(\PMC)}(\CFDDa(-\psi),\CFDa(\phi))
\]
(see~\cite{LOT4}). Then one uses
Formula~\eqref{eq:hom-pairing}. Indeed, this algorithm has already
been implemented by Lipshitz-Ozsv\'ath-Thurston~\cite{LOT4} and
Zhan~\cite{Zhan:bfh}.

Recall that a \DA\ bimodule $\lsup{\Blg}P_\Alg$ is called \emph{quasi-invertible} if there is a \DA\ bimodule $\lsup{\Alg}Q_\Blg$ so that
\[
  \lsup{\Blg}P_\Alg\DT\lsup{\Alg}Q_\Blg\simeq \lsup{\Blg}[\Id_\Blg]_\Blg \qquad\text{ and }\qquad
  \lsup{\Alg}Q_\Blg\DT\lsup{\Blg}P_\Alg\simeq \lsup{\Alg}[\Id_\Alg]_\Alg.
\]

Let $\Mor^{\Blg}(\lsup{\Blg}P_\Alg,\lsup{\Blg}P_\Alg)$ denote the
complex of left type $D$ morphisms of $P$. This morphism complex is an
$\Alg$-bimodule. (The module structure is somewhat intricate;
see~\cite[Section 2.3.4]{LOT2}.)

We have the following Yoneda lemma:
\begin{lemma}\label{lem:Yoneda}
  Let $\Alg$ and $\Blg$ be \dg algebras and $\lsup{\Blg}P_\Alg$ a
  quasi-invertible \DA\ bimodule. Then there is a quasi-isomorphism of
  $\Ainf$-bimodules
  \[
    \Omega\co\lsub{\Alg}\Alg_\Alg\stackrel{\simeq}{\longrightarrow}\Mor^{\Blg}(\lsup{\Blg}P_\Alg,\lsup{\Blg}P_\Alg)
  \]
  which sends the multiplicative identity $1\in\Alg$ to the identity
  morphism $\Id_P$. More generally, the $\Ainf$-bimodule map $\Omega$ is
  given by
  \begin{equation}\label{eq:Yoneda-is-hard}
    \Omega_{m,1,n}(a_1,\dots,a_m,a,a'_1,\dots,a'_n)(x)=\delta^1_{1+m+1+n}(x,a_1,\dots,a_m,a,a'_1,\dots,a'_n)
  \end{equation}
  where $\delta^1_k$ is the structure map of $P$.
\end{lemma}
\begin{proof}
  Let $\Mor^{\Alg}(\lsup{\Alg}[\Id_\Alg]_\Alg,\lsup{\Alg}[\Id_\Alg]_\Alg)$ be
  the chain complex of type $D$ structure morphisms. Then the map
  $F\co \Alg\to \Mor_{\Alg}(\lsup{\Alg}[\Id_\Alg]_\Alg,\lsup{\Alg}[\Id_\Alg]_\Alg)$
  defined by
  \begin{align*}
    F_1(a)&=(1\mapsto a\otimes 1)\\
    F_n&=0\qquad\qquad\qquad n>1
  \end{align*}
  is a chain homotopy equivalence. Next, since $P$ is
  quasi-invertible, the functor $P\DT\cdot$ is a quasi-equivalence of
  \dg categories. Thus, the map
  \begin{align*}
    G\co \Mor^{\Alg}(\lsup{\Alg}[\Id_\Alg]_\Alg,\lsup{\Alg}[\Id_\Alg]_\Alg)&\to\Mor^{\Blg}(P\DT[\Id_\Alg],P\DT[\Id_\Alg]),\\
    G(f)&=\Id_P\DT f
  \end{align*}
  is a quasi-isomorphism. (Compare~\cite[Proposition 2.3.36]{LOT2}.)
  The composition $G\circ F$ is the desired equivalence. Tracing
  through the definitions gives the Formula~\eqref{eq:Yoneda-is-hard}.
\end{proof}

\begin{corollary}\label{cor:our-Yoneda}
  Under the identification
  \[
    \overline{\CFDAa(\AZ)}\DT\CFAAa(\AZ)\DT\CFDAa(\AZ)\simeq\Mor^{\Alg(\PMC)}(\CFDAa(\AZ),\CFDAa(\AZ)),
  \]
  the unique homogeneous homotopy equivalence (of $\Ainf$-bimodules)
  \[
    \Alg(\PMC)=\CFAAa(\AZ)\stackrel{\Omega}{\longrightarrow}\Mor^{\Alg(\PMC)}(\CFDAa(\AZ),\CFDAa(\AZ))
  \]
  is given by 
  \begin{align*}
    \Omega_1(a)(x)&=\delta^1_2(x,a)\\
    \Omega_n(a_1,\dots,a_n)&=0 \qquad\qquad\qquad n>1.
  \end{align*}
\end{corollary}
\begin{proof}
  This is immediate from Lemma~\ref{lem:Yoneda} and the fact that the
  structure map $\delta^1_n$ for $\CFDAa(\AZ)$ vanishes for $n>2$.
\end{proof}

\begin{theorem}\label{thm:Mor-iota}
  Fix a Heegaard splitting $Y=(-\HB_0)\cup\HB_1$ of $Y$. Then up to
  homotopy the map $\iota\co \CFa(Y)\to\CFa(Y)$ is given by the
  composition
    {\small
  \begin{align*}
    \Mor^{\Alg(\PMC)}(\CFDa(\HB_0),\CFDa(\HB_1))
    &\stackrel{\Omega}{\longrightarrow} 
    \Mor^{\Alg(\PMC)}(\CFDAa(\AZ)\DT\CFDa(\HB_0),\CFDAa(\AZ)\DT\CFDa(\HB_1))\\
    &\stackrel{\Psi}{\longrightarrow}
      \Mor^{\Alg(\PMC)}(\CFDa(\HB_0),\CFDa(\HB_1))
  \end{align*}}
where
{\small
  \[
    \Omega\co \Mor^{\Alg(\PMC)}(\CFDa(\HD_0),\CFDa(\HD_1))\to \Mor^{\Alg(\PMC)}(\CFDAa(\AZ)\DT\CFDa(\HD_0),\CFDAa(\AZ)\DT\CFDa(\HD_1))
  \]}
  sends a morphism $f$ to $\Id_{\CFDAa(\AZ)}\DT f$ and, if
  $\Psi_i\co \CFDAa(\AZ)\DT\CFDa(\HD_0)\to \CFDa(\HD_i)$ is the homogeneous
  homotopy equivalence, then $\Psi$ sends a morphism
  $g\in
  \Mor^{\Alg(\PMC)}(\CFDAa(\AZ)\DT\CFDa(\HD_0),\CFDAa(\AZ)\DT\CFDa(\HD_1))$
  to $\Psi_1\circ g\circ \Psi_0^{-1}$.
\end{theorem}
This seems to be a succinct, and computer-friendly, description of the
map $\iota$.
\begin{proof}
  Choose a Heegaard diagram $\HD_i$ for $\HB_i$. Then the pairing theorem gives
  \[
    \Mor^{\Alg(\PMC)}(\CFDa(\HB_0),\CFDa(\HB_1))\simeq
    \CFa((-\HD_0)\cup\AZ\cup\HD_1)
  \]
  which is identified, via, $\eta$, with
  $\CFa(\HD_0^\beta\cup\bAZ^\beta\cup\conj{\HD_1}).$

  Similarly,
  \[
    \Mor^{\Alg(\PMC)}(\CFDAa(\AZ)\DT\CFDa(\HB_0),\CFDAa(\AZ)\DT\CFDa(\HB_1))\simeq
    \CFa(\HD_0^\beta\cup\AZ^\beta\cup\bAZ^\beta\cup\bAZ^\beta\cup\conj{\HD_1}).
  \]
  
  Consider a sequence of Heegaard moves
  \begin{align*}
    \HD_0^\beta\cup\bAZ^\beta\cup\conj{\HD_1}&\to(\HD_0^\beta\cup\AZ^\beta)\cup(\bAZ^\beta\cup\bAZ^\beta\cup\conj{\HD_1})\\
    &\to 
      (-\HD_0)\cup(\AZ\cup\HD_1)
  \end{align*}
  where the first arrow does not change the diagrams at the end and
  the second arrow consists of bordered Heegaard moves changing the
  diagrams on the two sides of the big union sign. There are two
  associated maps on $\CFa$. By the pairing theorem for triangles, the first map is induced by a map
  \[
    \Alg(\PMC)=\CFAAa(\bAZ^\beta)\to \CFAAa(\AZ^\beta\cup\bAZ^\beta\cup\bAZ^\beta)\simeq
    \Mor^{\Alg(\PMC)}(\CFDAa(\AZ),\CFDAa(\AZ)).
  \]
  By uniqueness, this map is the map $\Omega$ of
  Corollary~\ref{cor:our-Yoneda}. It follows from the definition of
  $\Omega$ and the pairing theorem that the induced map
    \[
    \Mor^{\Alg(\PMC)}(\CFDa(\HD_0),\CFDa(\HD_1))\to \Mor^{\Alg(\PMC)}(\CFDAa(\AZ)\DT\CFDa(\HD_0),\CFDAa(\AZ)\DT\CFDa(\HD_1))
  \]
  sends $f$ to $\Id\DT f$.  Similarly, by the pairing theorem for
  triangles, the second map is induced by an equivalence on each of
  the parenthesized pieces, and thus agrees with the map $\Psi$.
\end{proof}

The mapping class group action admits a similar description:  the action of $\chi$ is given by 
\begin{align*}
  \Mor^{\Alg(\PMC)}&(\CFDa(H_g,\phi_0),\CFDa(H_g,\phi_0\circ\psi))\\
  &\to \Mor^{\Alg(\PMC)}(\CFDAa(\chi^{-1})\DT\CFDa(H_g,\phi_0),\CFDAa(\chi^{-1})\DT\CFDa(H_g,\phi_0\circ\psi))\\
  &\to \Mor^{\Alg(\PMC)}(\CFDa(H_g,\phi_0),\CFDa(H_g,\phi_0\circ\psi))
\end{align*}
where the first map sends a morphism $f$ to $\Id\DT f$ and the second
sends $g$ to $\Theta_1\circ g\circ\Theta_0^{-1}$. We can rewrite this
using $\CFDDa(\chi)$ instead of $\CFDAa(\chi)$ as
\begin{align*}
  \Mor&{}^{\Alg(\PMC)}(\CFDa(H_g,\phi_0),\CFDa(H_g,\phi_0\circ\psi))\\
  &\to \Mor^{\Alg(\PMC)}\bigl(\Mor^{\Alg(\PMC)}(\CFDDa(\chi),\CFDa(H_g,\phi_0)),\Mor^{\Alg(\PMC)}(\CFDDa(\chi),\CFDa(H_g,\phi_0\circ\psi))\bigr)\\
  &\to \Mor^{\Alg(\PMC)}(\CFDa(H_g,\phi_0),\CFDa(H_g,\phi_0\circ\psi))
\end{align*}
where the first arrow sends a morphism $f$ to the morphism which sends
a morphism $h$ to $f\circ h$ and the second arrow is again induced by
the unique homotopy equivalences
$\Mor(\CFDDa(\chi),\CFDa(H_g,\phi_0))\simeq \CFDa(H_g,\phi_0)$ and
$\Mor(\CFDDa(\chi),\CFDa(H_g,\phi_0\circ\psi))\simeq
\CFDa(H_g,\phi_0\circ\psi)$. The proof that this gives the mapping
class group action is similar to the proof of
Theorem~\ref{thm:Mor-iota} and is left to the reader.

\section{Examples}\label{sec:examples}
For a knot $K$ in $S^3$, let $\Sigma(K)$ denote the branched double cover of $K$.  To
illustrate the algorithm for computing $\iota$, we finish the computation of $\HFIa(\Sigma(K))$
for knots $K$ through $10$ crossings.

If $\Sigma(K)$ is an $L$-space then, since $\Sigma(K)$ is a rational homology sphere with a unique
$\Spin$-structure, $\HFIa(\Sigma(K))\cong \Field^{\det(K)+1}$. That is, $\HFIa(\Sigma(K))$ has two generators for each conjugacy class of $\SpinC$-structures. The $Q$-action takes one generator corresponding to the $\Spin$-structure to the other, and vanishes on all other generators. All knots $K$
with $9$ or fewer crossings have $\Sigma(K)$ an $L$-space. Indeed, except for
$8_{19}=T(3,4)$, $9_{42}$ and $9_{46}$, every knot $K$ with $9$ or fewer crossings is
quasi-alternating~\cite{JS:quasi-alternating,Jablan:quasi-alternating}; for quasi-alternating
knots, $\Sigma(K)$ is an $L$-space~\cite{BrDCov}.
It turns out that $\Sigma(8_{19})$, $\Sigma(9_{42})$
and $\Sigma(9_{46})$ are $L$-spaces. (This can be checked using Zhan's computer
program~\cite{Zhan:bfh}.)

The $10$-crossing knots $K$ for which $\Sigma(K)$ is not an $L$-space are
listed in Table~\ref{tab:comps}. The computation of which of these
spaces are not $L$-spaces, and the dimensions of their Floer
homologies, was accomplished by Zhan. Computation of $\HFIa$ for
these manifolds was carried out by a modest extension of Zhan's
program, using the algorithm described above. The first two knots,
$10_{139}$ and $10_{145}$, are Montesinos knots, hence our our
computation is implied by (and agrees with) the computation of $\HFIm$
for Seifert fibered spaces~\cite{DM:involutive-plumbed}.  We
make a few further comments about the details of our implementation
below.

\begin{table}
  \centering
  \begin{tabular}{cccc}
    \toprule
    Knot $K$ & $\det(K)$ & $\dim\HFa(\Sigma(K))$ & $\dim\HFIa(\Sigma(K))$\\
    \midrule
    $10_{139}$ &3 &5 & 6 \\
    $10_{145}$ & 3& 5& 6\\
    $10_{152}$ & 11 & 13 & 14\\
    $10_{153}$ & 1 & 5 & 6\\
    $10_{154}$ & 13 & 15 & 16\\
    $10_{161}$ & 5 & 7 & 8\\
    \bottomrule\\
  \end{tabular}
  \caption{\textbf{The $10$-crossing knots with $\Sigma(K)$ not an $L$-space.}
    The table lists the dimensions of $\HFa(\Sigma(K))$ and $\HFIa(\Sigma(K))$,
    as computed by Zhan's program and its extension, for these knots, as well as
    $\det(K)=|H_1(\Sigma(K))|$. Computations of $\det(K)$ are taken from
    The Knot Atlas, katlas.org.}
  \label{tab:comps}
\end{table}

Both Zhan's code and our extension, which is now included in Zhan's package~\cite{Zhan:bfh}, are written in Python (version 2.7). Zhan's
code includes classes for chain complexes, type $D$ structures, and type \DA\
structures, as well as for morphisms between them. He also, of course,
implemented basic operations on these structures, including taking the box
tensor product of a type $D$ structure and a type \DA\ structure and computing
the morphism complex between two type $D$ structures. His program also automates
computation of $\HFa(\Sigma(K))$ given a bridge diagram for $K$. The algorithms
behind Zhan's code use properties of the bordered bimodules which appear only in
his thesis~\cite{Zhan14:thesis} to compute tensor products without writing down
all of the generators. (He calls this technique \emph{extending by the identity}
and the local objects that he extends \emph{local type \DA\ structures}.) The
upshot is that his code computes $\CFDa(\HB_0)$ and $\CFDa(\HB_1)$ efficiently.

In our extension, we implemented the bimodule $\CFDAa(\AZ)$, mapping cones of
maps between type $D$ structures and chain complexes, composition of morphisms
between type $D$ structures, and the tensor product of a morphism of type $D$
structures with the identity map of a type \DA\ structure. Computing mapping
cones gives some easy sanity checks: it makes testing whether maps are
quasi-isomorphisms trivial, by checking whether their mapping cones are
acyclic.

Our code computes the rank of $\HFIa$ by:
\begin{enumerate}
\item Computing $\CFDa(\HB_0)$, $\CFDa(\HB_1)$, $\CFDAa(\AZ)\DT\CFDa(\HB_0)$,
  and $\CFDAa(\AZ)\DT\CFDa(\HB_1)$, as well as various morphism complexes between them.
\item Computing a basis $\{f_1,\dots,f_n\}$ for
  $H_*\Mor(\CFDa(\HB_0),\CFDa(\HB_1))$, consisting of explicit cycles in
  $\Mor(\CFDa(\HB_0),\CFDa(\HB_1))$.
\item For each basis element $f_i$, computing $\Id_{\CFDAa(\AZ)}\DT f_i$. 
\item Computing a basis for $H_*\Mor(\CFDa(\HB_0),\CFDAa(\AZ)\DT\CFDa(\HB_0))$
  and for $H_*\Mor(\CFDAa(\AZ)\DT\CFDa(\HB_1),\CFDa(\HB_1))$.  Even
  though we do not implement the grading for $\CFDAa(\AZ)$, the way that Zhan's
  code computes homology automatically gives bases of homogeneous elements. Each
  of these bases has $2^k$ elements where $k$ is the genus of the Heegaard
  splitting. For the computations in Table~\ref{tab:comps}, $k=2$, so each of
  these bases has $4$ elements.
\item Searching through these bases to find the unique homotopy equivalences
  $\Psi_0^{-1}$ and $\Psi_1$.
\item For each $f_i$, computing the composition $\Psi_1\circ \bigl(\Id_{\CFDAa(\AZ)}\DT f_i\bigr)\circ \Psi_0^{-1}$. The map
  \[
    f_i\mapsto \Psi_1\circ \bigl(\Id_{\CFDAa(\AZ)}\DT f_i\bigr)\circ \Psi_0^{-1}
  \]
  is a map
  \[
    [H_*\Mor(\CFDa(\HB_0),\CFDa(\HB_1))]\to \Mor(\CFDa(\HB_0),\CFDa(\HB_1))
  \]
  representing $\iota$. (Mapping from the homology of the complex to the complex
  means we do not have to choose a projection from the morphism complex to its
  homology.) Abusing notation, we call this map $\iota$.
\item There is also an inclusion
  \[
    \Id\co H_*\Mor(\CFDa(\HB_0),\CFDa(\HB_1))\to \Mor(\CFDa(\HB_0),\CFDa(\HB_1))
  \]
  induced by the choice of cycles $f_1,\dots,f_n$. The involutive Floer homology
  is then the homology of $\Cone(\iota+\Id)$.
\end{enumerate}

The computations in Table~\ref{tab:comps} are fairly slow: on a circa 2016
MacBook Pro with 16 GB of RAM the code computes $\HFa(\Sigma(K))$ within a few
minutes but each computation of $\HFIa(\Sigma(K))$ takes up to several hours. (We
have not made a serious attempt to improve the efficiency of our code.)

\subsection{Computing \texorpdfstring{$\HFIm$}{HFI-minus} from \texorpdfstring{$\HFIa$}{HFI-hat}}

Sometimes, one can recover $\HFIm(Y)$ from $\HFa(Y)$ and
$\HFIa(Y)$. (This is desirable given that most known applications use $\HFIm(Y)$ or $\HFI^+(Y)$ rather than $\HFIa(Y)$.) We illustrate the process of recovering $\HFIm$ by computing $\HFIm(\Sigma(10_{161}))$
up to a grading shift. 

Let $\spinc_0$ denote the $\Spin$-structure on $\Sigma(10_{161})$. If
$\spinc\in\spinC(\Sigma(10_{161}))$ is any other $\SpinC$-structure
then, since $\HFa(\Sigma(10_{161}),\spinc)\cong\Field$,
$\HFI^-(\Sigma(10_{161}),[\spinc])\cong \Field[U]\oplus\Field[U]$ with
trivial $Q$-action, where $[\spinc]$ denotes the orbit consisting of the $\SpinC$ structure and its conjugate. So, for the rest of the section we focus on
$\HFI^-(\Sigma(10_{161}),\spinc_0)$.

\begin{lemma} \label{lem:minus-from-hat} Let $d = d(\Sigma(10_{161}, \spinc_0))$ be the Heegaard Floer correction term of the $\SpinC$-structure $\spinc_0$ on $\Sigma(10_{161})$. Then
  \[
    \HFIm(\Sigma(10_{161}), \spinc_0) \simeq \Field[U]_{(d-3)}\langle a \rangle\oplus \Field[U]_{(d-2)} \langle b \rangle \oplus (\Field)_{(d-2)}\langle c \rangle
  \]
  with $Q$-action given by $Qa = Ub$ and $Qb=Qc=0$.
\end{lemma}

In \cite{HM:involutive}, C.~Manolescu and the first author extract two invariants of $\Field$-homology cobordism from involutive Heegaard Floer homology, called the \emph{involutive correction terms}. Given a rational homology sphere $Y$ and a conjugation-invariant $\SpinC$-structure $\spinc$, in terms of the minus variant, these invariants are
\[ \dl(Y,\spinc) = \max \{r \mid \exists \ x \in \HFIm_r(Y, \spinc), \forall \ n, \ U^nx\neq 0 \ \text{and} \ U^nx \notin \operatorname{Im}(Q)\} + 1 \]
and
\[ \du(Y,\spinc) = \max \{r \mid \exists \ x \in \HFIm_r(Y,\spinc),\forall \ n, U^nx\neq 0; \exists \ m\geq 0 \ \operatorname{s.t.} \ U^m x \in \operatorname{Im}(Q)\} +2. \]

We therefore have the following corollary of Lemma~\ref{lem:minus-from-hat}:

\begin{corollary}
The involutive correction terms of $\Sigma(10_{161})$ in the unique $\Spin$ structure are related to $d=d(\Sigma(10_{161}), \spinc_0)$ by
\begin{align*}
\dl(\Sigma(10_{161}), \spinc_0) &= d-2 \\
\du(\Sigma(10_{161}), \spinc_0) &= d.
\end{align*}
\end{corollary}

\begin{proof}[Proof of Lemma \ref{lem:minus-from-hat}] Let $K =
  10_{161}$. Zhan's code for computing $\HFa(\Sigma(K))$ can be used
  to compute relative gradings and $\SpinC$-structures for generators
  of $\HFa(\Sigma(K))$. Arbitrarily numbering the $\SpinC$-structures
  of the generators by $0,\dots,4$, the code finds that, up to a shift, the gradings
  of the generators representing the different $\SpinC$-structure are:
  \begin{center}
    \begin{tabular}{ccccc}
      \toprule
    $\spinc$&$\gr$&\ \ \ \  &$\spinc$&$\gr$\\
    \midrule
    3 & 6/5 & & 0 & 2/5\\ 
    3 & 6/5 & & 1 & 2/5\\
    3 & 1/5 & & 2 & 0\\
            & & & 4 & 0\\
      \bottomrule
  \end{tabular}
\end{center}

Thus, the $\SpinC$-structure labeled $3$ must be the central $\SpinC$-structure. From the computer computation, $\rank(\HFIa(\Sigma(10_{161})))=8=\rank(\HFa(\Sigma(10_{161})))+1$, so
$\iota_*$ must have exactly one fixed point, which must be the generator in relative grading $1/5$. The other two elements in this $\SpinC$-structure must, up to a change of basis, be interchanged by $\iota_*$. We conclude that $\HFa(\Sigma(K), \spinc_0)$ contains three elements, two in some grading $q$ and one in grading $q-1$, and that up to a change of basis, the two elements in grading $q$ are interchanged by $\iota_*$. 

Now, recall that there is a long exact sequence
\begin{equation} \label{eq:minus-exact}
\cdots \to \HF^-(\Sigma(K)) \xrightarrow{\cdot U} \HF^-(\Sigma(K)) \to \HFa(\Sigma(K)) \to \HF^-(\Sigma(K))\to \cdots 
\end{equation}
such that the map $\HF^-(\Sigma(K))\to \HFa(\Sigma(K))$ increases the
grading by $2$ and the map $\HFa(\Sigma(K))\to\HF^-(\Sigma(K))$
decreases the grading by $1$ \cite[Proposition
2.1]{OS04:HolDiskProperties}. This long exact sequence commutes at
every step with $\iota_*$ \cite[Proof of Proposition
4.1]{HM:involutive}. (Strictly speaking, this was proved for the analogous sequence for $\HF^+$, but the proof for $\HF^-$ is identical.) It follows from the existence of this long exact
sequence that there is a noncanonical isomorphism
$\HF^-(\Sigma(K)) \simeq \Field[U]\langle \alpha \rangle \oplus \Field
\langle \beta \rangle$, where both $\alpha$ and $\beta$ lie in grading
$q-2$. In particular, the ordinary Heegaard Floer correction term is
$d(\Sigma(K), \spinc_0) = q$. Further, the grading shifts imply that
the summand of $\HFa(\Sigma(K))$ in grading $q$ is precisely the image
of the summand of $\HF^-(\Sigma(K))$ in grading $q-2$, which is
spanned as a vector space by $\alpha$ and $\beta$. Therefore since the
long exact sequence~(\ref{eq:minus-exact}) respects the action
$\iota_*$, the involution on $\HFa(\Sigma(K))$ is determined by the
involution on $\HF^-(\Sigma(K))$. There are exactly two
$U$-equivariant involutions on $\HF^-(\Sigma(K))$: the identity and
the involution $\iota_*(\alpha) = \alpha + \beta$,
$\iota_*(\beta)=\beta$. The first of these induces the identity
involution on $\HFa(\Sigma(K))$, contradicting the computer
computation. Thus, $\iota_*(\alpha) = \alpha + \beta$,
$\iota_*(\beta)=\beta$.

Recall that there is an exact triangle 
\begin{equation}
\label{pic:exact1}
\begin{tikzpicture}[baseline=(current  bounding  box.center)]
\node(1)at(0,0){$\HFIm(Y,\spinc)$};
\node(2)at (-2,1){$\HF^-(Y, \spinc)$};
\node(3)at (2,1){$Q \cdot \HF^-(Y,\spinc)[-1]$};
\path[->](2)edge node[above]{$Q(1+ \iota_*)$}(3);
\path[->](3)edge (1);
\path[->](1)edge(2);
\end{tikzpicture}
\end{equation}
\cite[Proposition 4.6]{HM:involutive}.

Ordinarily, the existence of this triangle is insufficient to determine $\HFIm(Y, \spinc)$. (That is, $\HFIm$ is in general not a mapping cone of the map $1+\iota_*$ on $\HF^-$, unlike the hat variant.) However, in this case the complex is sufficiently small that given our computation of $\iota_*$, the mapping cone of $(1+\iota_*)$ is the unique $\Field[U,Q]/(Q^2)$-module that fits into the long exact triangle.  The map $Q(1+ \iota_*)$ takes $\alpha$ to $Q\beta$. So, $\HFIm(\Sigma(K),\spinc_0)$ is generated by $U \alpha = a$, $Q \alpha = b$, and $\beta =c$, and those elements lie in gradings $d-3$, $d-2$, and $d-2$ respectively.
\end{proof}

\begin{remark} The reader may have noticed that the complex $\HF^-(\Sigma(10_{161}), \spinc_0)$ is (after a change of basis) a symmetric graded root. Indeed, I.~Dai and C.~Manolescu recently showed that whenever $(Y, \spinc)$ is such that $\HF^-(Y, \spinc)$ is a symmetric graded root with involution given by the canonical symmetry, $\HFIm(Y, \spinc)$ is a mapping cone on $\HF^-(Y,\spinc)$ \cite[Theorem 1.1]{DM:involutive-plumbed}.
\end{remark}

\begin{remark}
  It may be interesting to compare these computations with Lin's spectral sequence from a variant of Khovanov homology to involutive monopole Floer homology of the branched double cover~\cite{Lin:involutive-Kh}.
\end{remark}

\begin{remark}
  One could call a rational homology sphere $Y$ \emph{$\HFIa$-trivial} if for each $\Spin$-structure $\spinc$ on $Y$, $\HFIa(Y,\spinc)\cong\HFa(Y,\spinc)\oplus\Field$ where $Q\cdot\HFa(Y,\spinc)=0$ and $Q$ is non-vanishing on the remaining generator. At the time of writing, no $\HFIa$-nontrivial rational homology sphere $Y$ is known.
\end{remark}



\bibliographystyle{halpha-abbrv.bst}
\bibliography{heegaardfloer}

\newpage
\newcommand{\nocontentsline}[3]{}
\newcommand{\tocless}[2]{\bgroup\let\addcontentsline=\nocontentsline#1{#2}\egroup}
\section{Corrections}
\subsection{Correction regarding the pairing theorem for triangles}
In the proof of Lemma 5.6 of~\cite{HL19:involutive}, we wrote:
\begin{quote}
  For handleslides, both horizontal maps are defined by counting holomorphic triangles, and the fact that this diagram commutes up to homotopy is a special case of the pairing theorem for triangles [LOT14a].
\end{quote}

This is not correct. For $\beta$-handleslides, the statement and proof are correct. For $\alpha$-handleslides, an additional argument is needed. As we explain below, the simplest solution is to change Lemma 5.6 to say that associated to each bordered Heegaard move, there is a map of bordered modules so that the specified square commutes. For moves other than $\alpha$-handleslides, the map of bordered modules is the one from the invariance proof of bordered Floer homology~\cite{LOT1}. For $\alpha$-handleslides, the map is constructed indirectly, by gluing on the ($\alpha,\beta$)-bordered Heegaard diagram $\AZ$, performing the handleslide, and then removing $\AZ$. We explain this in more detail below.

After substituting this new handleslide map for the usual one throughout the paper, all other results and proofs are unchanged.

\subsubsection{More details about the corrected lemma}\label{sec:second-resolution}

Let $\HD=(\Sigma,\alphas,\betas)$ be a bordered Heegaard diagram and $\HD^H=(\Sigma,\alphas^H,\betas)$
the result of handlesliding an $\alpha$-curve over an $\alpha$-circle. 
Let $\AZ$ be the Auroux-Zarev diagram, an $(\alpha,\beta)$-bordered Heegaard diagram, and let $\AZbar$ be the dual diagram, so $\AZ\cup\AZbar$ is equivalent to the identity diagram. (Any other invertible $(\alpha,\beta)$-bordered Heegaard diagram would work just as well here.) The pairing theorem induces homotopy equivalences
\[
  \CFA(\HD\cup \AZ) \simeq \CFA(\HD)\DT\CFDA(\AZ)\qquad
  \CFA(\HD^H\cup \AZ) \simeq \CFA(\HD^H)\DT\CFDA(\AZ),
\]
and similarly for $\CFD$. The diagrams $\HD\cup \AZ$ and $\HD^H\cup \AZ$ are $\beta$-bordered. Since $\CFDA(\AZ)$ is invertible, tensoring with it is an equivalence of categories. Instead of defining the handleslide invariance map from $\CFA(\HD)$ to $\CFA(\HD^H)$ by counting triangles as in~\cite{LOT1}, define it to be the unique map (up to homotopy) $f\colon \CFA(\HD)\to\CFA(\HD^H)$ so that
\[
  \begin{tikzcd}
    \CFA(\HD)\DT\CFDA(\AZ)\arrow[r,"f\DT\Id"]\arrow[d,"\simeq"] & \CFA(\HD^H)\DT\CFDA(\AZ)\arrow[d,"\simeq"]\\
    \CFA(\HD\cup\AZ)\arrow[r,"g"] & \CFA(\HD^H\cup\AZ)
  \end{tikzcd}
\]
homotopy commutes, where the bottom map $g$
is induced by counting holomorphic triangles. Here, the vertical homotopy equivalences are the ones induced by the pairing theorem. (Note that $g$ is defined by counting triangles with respect to two sets of circles and one set of arcs, since $\HD\cup\AZ$ is $\beta$-bordered.)

To verify this case of~\cite[Lemma 5.6]{HL19:involutive}, fix another bordered Heegaard diagram $\HD'$. The claim is that the diagram
\[
\begin{tikzcd}
  \CFA(\HD)\DT\CFD(\HD')\arrow[r,"f\DT\Id"]\arrow[d,"\simeq"] & \CFA(\HD^H)\DT\CFD(\HD')\arrow[d,"\simeq"]\\
  \CFa(\HD\cup\HD') \arrow[r] & \CFa(\HD^H\cup\HD')
\end{tikzcd}  
\]
homotopy commutes, where the vertical arrows are induced by the pairing theorem and the bottom arrow is Ozsv\'ath-Szab\'o's handleslide map. Consider the larger diagram
{\small
\[
\begin{tikzcd}
  \CFA(\HD)\DT\CFD(\HD')\arrow[r,"f\DT\Id"]\arrow[d,"\simeq"] & \CFA(\HD^H)\DT\CFD(\HD')\arrow[d,"\simeq"]\\
  \CFA(\HD)\DT\CFDA(\AZ)\DT\CFDA(\AZbar)\DT\CFD(\HD')\arrow[r,"f\DT\Id^3"] \arrow[d,"\simeq"]& 
  \CFA(\HD^H)\DT\CFDA(\AZ)\DT\CFDA(\AZbar)\DT\CFD(\HD')\arrow[d,"\simeq"]\\
  \CFA(\HD\cup\AZ)\DT\CFDA(\AZbar\cup\HD')\arrow[r,"g\DT\Id"] \arrow[d,"\simeq"]& 
  \CFA(\HD^H\cup\AZ)\DT\CFDA(\AZbar\cup\HD')\arrow[d,"\simeq"]\\
  \CFa(\HD\cup\AZ\cup\AZbar\cup\HD')\arrow[r] \arrow[d,"\simeq"]& 
  \CFa(\HD^H\cup\AZ\cup\AZbar\cup\HD')\arrow[d,"\simeq"]\\
  \CFa(\HD\cup\HD') \arrow[r] & \CFa(\HD^H\cup\HD').
\end{tikzcd}  
\]}Here, the top vertical arrows are induced by a homotopy equivalence between $\CFDA(\AZ)\DT\CFDA(\AZbar)$ and the identity bimodule. (This homotopy equivalence is, in fact, unique up to homotopy~\cite[Corollary 4.6]{HL19:involutive}.) The bottom vertical arrows are induced by a sequence of Heegaard moves (from $\AZ\cup\AZbar$ to the identity diagram followed by some destabilizations). The other vertical arrows come from the pairing theorem, or are identity maps. The bottom two horizontal arrows are Ozsv\'ath-Szab\'o's handleslide maps.

The top square homotopy commutes because $\DT$ is, up to homotopy, a bifunctor~\cite[Lemma 2.3.13]{LOT2}.
The second square commutes from the definition of $g$. The third square homotopy commutes by the pairing theorem for triangles (the version proved in~\cite{LOT:DCov2}). The bottom square homotopy commutes by naturality of $\CFa$~\cite{OS06:HolDiskFour,JT:Naturality}: the two ways around the square correspond to doing a handleslide and then a sequence of Heegaard moves in a different part of the diagram. This proves the result.

\tocless\subsubsection*{Acknowledgements.} We thank Ian Zemke for pointing out this mistake and for helpful comments on this note.

\subsection{Correction regarding certain commutative squares}

The proof of Theorem 7.1 of~\cite{HL19:involutive} begins with a large diagram. Regarding this diagram, we wrote: ``the fact that the top two rows commute on the nose follows from basic properties of
the box tensor product~\cite[Lemma 2.3.3]{LOT2}.''

Both rows are instances of the commutative diagram in~\cite[Case (2) of Lemma 2.3.3]{LOT2}, or rather its generalization to bimodules (mentioned slightly later in~\cite[Section 2.3.2]{LOT2}). In general, that diagram commutes only up to homotopy, so this claim needs further justification, and in fact is not correct without a further assumption.
 
The second row does commute identically. The homotopy for the left square of the second row is given graphically by 
\tikzset{taa/.style={double, double equal sign distance, -implies}} 
\[
\begin{tikzpicture}
  \node at (0,0) (tc) {};
  \node at (-2,0) (tl) {};
  \node at (2,0) (tr) {};
  \node at (2,-2) (delta1) {$\delta$};
  \node at (2,-3) (phi) {$\phi^1$};
  \node at (2,-4) (delta2) {$\delta$};
  \node at (2,-7) (br) {};
  \node at (0,-5) (middelta) {$\delta$};
  \node at (-2,-6) (Psi) {$\Psi_0$}; 
  \node at (-2,-7) (bl) {};
  \node at (0,-7) (bc) {};
  \draw[->] (tl) to node[below,sloped]{\tiny$\CFA(X(K))\DT\CFDA(\AZbar)$} (Psi);
  \draw[->] (Psi) to (bl);
  \draw[->] (tc) to node[below,sloped]{\tiny$\CFDA(\AZ)$} (middelta);
  \draw[->] (middelta) to (bc);
  \draw[->] (tr) to node[below,sloped]{\tiny$\CFD(\mathcal{H}_\infty)$} (delta1);
  \draw[->] (delta1) to (phi);
  \draw[->] (phi) to (delta2);
  \draw[->] (delta2) to (br);
  \draw[taa] (delta1) to (middelta);
  \draw[taa] (middelta) to (Psi);
  \draw[taa] (delta2) to (middelta);
  \draw[->] (phi) to (middelta);
\end{tikzpicture}.
\]
From the form of $\CFDA(\AZ)$, this homotopy vanishes: the operation $\delta^1$ on $\CFDA(\AZ)$ with at least one algebra input outputs an idempotent and with two or more algebra inputs vanishes, so the input to $\Psi_0$ in the diagram has at least one idempotent, and thus the output of $\Psi_0$ vanishes. The right square is similar.
 
The first row does not necessarily commute. Below, we give two ways to resolve the problem and obtain the stated result.

Before doing so, we additionally note that, in the bottom row of that commutative diagram, there is a missing associator homotopy. We explain this for the left square; the right square is similar. The map $G$ is a homotopy between $\Psi_1\circ(\Id\DT\phi)$ and $\phi\circ \Psi_1$. However, $(\Id\DT\Psi_1)\circ(\Id^2\DT\phi)$ is not equal to $\Id\DT(\Psi_1\circ(\Id\DT\phi))$, but rather homotopic to it: the homotopy is given by 
\[
\begin{tikzpicture}
\node at (-1,1) (tl) {};
\node at (1,1) (tr) {};
\node at (1,0) (delta0) {$\delta$};
\node at (1,-1) (phi) {$\Id\DT\phi$};
\node at (1,-2) (delta1) {$\delta$};
\node at (1,-3) (Psi) {$\Psi_1$};
\node at (1,-4) (delta2) {$\delta$};
\node at (-1,-5) (mu) {$\mu$};
\node at (-1,-6) (bl) {};
\node at (1,-6) (br) {};
\draw[->] (tr) to (delta0);
\draw[->] (delta0) to (phi);
\draw[->] (phi) to (delta1);
\draw[->] (delta1) to (Psi);
\draw[->] (Psi) to (delta2);
\draw[->] (delta2) to (br);
\draw[->] (tl) to (mu);
\draw[->] (mu) to (bl);
\draw[->] (phi) to (mu);
\draw[->] (Psi) to (mu);
\draw[taa, bend right=15] (delta0) to (mu);
\draw[taa] (delta1) to (mu);
\draw[taa] (delta2) to (mu);
\end{tikzpicture}.
\]
(See~\cite[Lemma 2.3.3(1)]{LOT2}.)
Similarly, $(\Id\DT\phi)\circ(\Id\DT\Psi_1)$ and $\Id\DT(\phi\circ\Psi_1)$ are not equal, but merely homotopic. The arrow labeled $\Id\DT G$ should include both of these homotopies as terms.
(The analogous homotopies for earlier rows of the diagram vanish because $\CFDA(\AZ)$ has no operations with more than one algebra input.)

For reference later, call these new homotopies $A$ and $B$ for the left square, and $C$ and $D$ for the right one.

\subsubsection{Option 1: making everything strict}
We can arrange that the first row commutes by choosing appropriate representatives. The row is, in fact, a composition of two rows, the first of which introduces a copy of $[\Id]$ in the middle and the second of which applies the map $\Omega\co [\Id] \to \CFDA(\AZbar)\DT\CFDA(\AZbar)$, i.e., 
{\scriptsize
\begin{equation}\label{eq:2-rows}
\begin{tikzcd}
  \CFA(X(K))\DT\CFD(\HD_\infty)\arrow[r,"\Id\DT\phi"]\arrow[d] & 
  \CFA(X(K))\DT\CFD(\HD_{-1})\arrow[r,"\Id\DT\psi"]\arrow[d] &
  \CFA(X(K))\DT\CFD(\HD_0)\arrow[d]\\
  \CFA(X(K))\DT[\Id]\DT\CFD(\HD_\infty)\arrow[r,"\Id^2\DT\phi"]\arrow[d,"\Id\DT\Omega\DT\Id"] & 
  \CFA(X(K))\DT[\Id]\DT\CFD(\HD_{-1})\arrow[r,"\Id^2\DT\psi"]\arrow[d,"\Id\DT\Omega\DT\Id"] &
  \CFA(X(K))\DT[\Id]\DT\CFD(\HD_0)\arrow[d,"\Id\DT\Omega\DT\Id"]\\
  \begin{matrix}\CFA(X(K))\DT\CFDA(\AZbar)\\\DT\CFDA(\AZ)\DT\CFD(\HD_\infty)\end{matrix}\arrow[r,"\Id^3\DT\phi"]\arrow[d] & 
  \begin{matrix}\CFA(X(K))\DT\CFDA(\AZbar)\\\DT\CFDA(\AZ)\DT\CFD(\HD_{-1})\end{matrix}\arrow[r,"\Id^3\DT\psi"]\arrow[d] &
  \begin{matrix}\CFA(X(K))\DT\CFDA(\AZbar)\\\DT\CFDA(\AZ)\DT\CFD(\HD_0)\end{matrix}\arrow[d]\\
  \vdots & \vdots & \vdots
\end{tikzcd}
\end{equation}
}

The first of these rows commutes identically. For the second, the homotopy for the left square is given by
\begin{equation}\label{eq:the-htpy}
\vcenter{\hbox{
  \begin{tikzpicture}
  \node at (0,0) (tl) {};
  \node at (2,0) (tr) {};
  \node at (2,-2) (delta1) {$\delta$};
  \node at (2,-3) (phi) {$\phi^1$};
  \node at (2,-4) (delta2) {$\delta$};
  \node at (0,-5) (Omega) {$\Id\DT\Omega$};
  \node at (0,-6) (bl) {};
  \node at (2,-6) (br) {};
  \draw[taa] (delta1) to (Omega);
  \draw[taa] (delta2) to (Omega);
  \draw[->] (phi) to (Omega);
  \draw[->] (tl) to node[below,sloped]{\tiny$\CFA(X(K))\DT[\Id]$} (Omega);
  \draw[->] (Omega) to (bl);
  \draw[->] (tr) to node[below,sloped]{\tiny$\CFD(\HD_\infty)$} (delta1);
  \draw[->] (delta1) to (phi);
  \draw[->] (phi) to (delta2);
  \draw[->] (delta2) to (br);
\end{tikzpicture}}}.
\end{equation}
The right square is again similar.
 
Suppose we have chosen $\CFA(X(K))=\CFA(X(K))\DT[\Id]$ to be a projective \dg module, rather than a more general $\Ainf$-module. (For instance, we could take $\CFA(X(K))$ to be the modulification of $\CFD(X(K))$.) Then
$\CFA(X(K))\DT\CFDA(\AZbar)\DT\CFDA(\AZ)\cong \CFA(X(K))\DT\CFDD(\AZbar)\DT\CFAA(\AZ)=\CFA(X(K))\DT\CFDD(\AZbar)\DT \Alg(\PMC)$
is also a projective \dg module. So, up to homotopy, $\Id\DT\Omega$ is an honest \dg module map $\Omega'$ (by~\cite[Proposition 2.4.1]{LOT2}). If we use $\Omega'$ in place of $\Omega$, the homotopy vanishes, and these squares commute. The composition of each column is still $\iota$, which was only well-defined up to homotopy, and the rest of the proof is unchanged.

\subsubsection{Option 2: making everything coherent}
Consider the first row of Formula~\eqref{eq:2-rows}. As noted above, there are homotopies making this diagram homotopy commute; the homotopy for the first square is given in Formula~\eqref{eq:the-htpy} and the second homotopy is similar. Denote these homotopies by $K$ and $L$, respectively. The compositions $(\Id^3\DT\psi)\circ K$ and $L\circ(\Id^2\DT\phi)$ are also homotopic, via the homotopy $M$ given by

\[
  \begin{tikzpicture}
    \node at (0,0) (tl) {};
    \node at (2,0) (tr) {};
    \node at (2,-2) (delta1) {$\delta$};
    \node at (2,-3) (phi) {$\phi^1$};
    \node at (2,-4) (delta2) {$\delta$};
    \node at (2,-5) (psi) {$\psi^1$};
    \node at (2,-6) (delta3) {$\delta$};
    \node at (0,-7) (Omega) {$\Id\DT\Omega$};
    \node at (0,-8) (bl) {};
    \node at (2,-8) (br) {};
    \draw[taa] (delta1) to (Omega);
    \draw[taa] (delta2) to (Omega);
    \draw[taa] (delta3) to (Omega);
    \draw[->] (phi) to (Omega);
    \draw[->] (psi) to (Omega);
    \draw[->] (tl) to node[below,sloped]{\tiny$\CFA(X(K))\DT[\Id]$} (Omega);
    \draw[->] (Omega) to (bl);
    \draw[->] (tr) to node[below,sloped]{\tiny$\CFD(\HD_\infty)$} (delta1);
    \draw[->] (delta1) to (phi);
    \draw[->] (phi) to (delta2);
    \draw[->] (delta2) to (psi);
    \draw[->] (psi) to (delta3);
    \draw[->] (delta3) to (br);
  \end{tikzpicture}
\]

That is, we have a diagram of the form

{\footnotesize
\[
\begin{tikzcd}[row sep=4em]
  \CFA(X(K))\DT[\Id]\DT\CFD(\HD_\infty)\arrow[r,"\Id^2\DT\phi"]\arrow[d,"\Id\DT\Omega\DT\Id"]\arrow[dr,dashed,"K"]\arrow[drr,dotted,"M", near start] & 
  \CFA(X(K))\DT[\Id]\DT\CFD(\HD_{-1})\arrow[r,"\Id^2\DT\psi"]\arrow[d,"\Id\DT\Omega\DT\Id"]\arrow[dr,dashed,"L"] &
  \CFA(X(K))\DT[\Id]\DT\CFD(\HD_0)\arrow[d,"\Id\DT\Omega\DT\Id"]\\
  \begin{matrix}\CFA(X(K))\DT\CFDA(\AZbar)\\\DT\CFDA(\AZ)\DT\CFD(\HD_\infty)\end{matrix}\arrow[r,"\Id^3\DT\phi"] & 
  \begin{matrix}\CFA(X(K))\DT\CFDA(\AZbar)\\\DT\CFDA(\AZ)\DT\CFD(\HD_{-1})\end{matrix}\arrow[r,"\Id^3\DT\psi"] &
  \begin{matrix}\CFA(X(K))\DT\CFDA(\AZbar)\\\DT\CFDA(\AZ)\DT\CFD(\HD_0)\end{matrix}
\end{tikzcd}  
\]}
where the maps satisfy
\begin{align}
  \partial\circ K + K\circ \partial &= (\Id\DT\Omega\DT\Id)\circ (\Id^2\DT\phi)+(\Id^3\DT\phi)\circ(\Id\DT\Omega\DT\Id)\label{eq:K-htpy}\\
  \partial\circ L + L\circ \partial &= (\Id\DT\Omega\DT\Id)\circ (\Id^2\DT\psi)+(\Id^3\DT\psi)\circ(\Id\DT\Omega\DT\Id)\\
  \partial\circ M + M\circ \partial &= (\Id^3\DT\psi)\circ K + L\circ (\Id^2\DT\phi).\label{eq:M-htpy}
\end{align}

We continue the proof of Theorem 7.1, keeping track of these maps. Composing the maps in each column of the diagram at the diagram at the beginning of the proof of Theorem 7.1 gives Diagram (7.3), which we now draw as:
\[
(7.3')\qquad
\begin{tikzcd}[row sep=4em]
  0 \arrow[r] & \CFa(Y) \arrow[r,"i"]\arrow[d,"\Id+\iota"]\arrow[dr,dashed,"G'"] \arrow[drr,dotted,"M'",near start]& \CFa(Y_{-1}(K))\arrow[r,"p"]\arrow[d,"\Id+\iota", near start]\arrow[dr,"H'",dashed] & \CFa(Y_0(K))\arrow[r]\arrow[d,"\Id+\iota"] & 0\\
  0 \arrow[r] & \CFa(Y) \arrow[r,"i"] & \CFa(Y_{-1}(K))\arrow[r,"p"] & \CFa(Y_0(K))\arrow[r] & 0
\end{tikzcd}
\]
The meaning of $G'$ and $H'$ has changed, and there is a new map $M'$. Now,
\begin{align*}
  G' &= \bigl((\Id\DT G)+A+B\bigr)\circ (\Psi_0\DT\Id^2)\circ\Omega + (\Id\DT\Psi_1)\circ (\Psi_0\DT\Id^2)\circ K \\
  H' &= \bigl((\Id\DT H)+C+D\bigr)\circ (\Psi_0\DT\Id^2)\circ\Omega + (\Id\DT\Psi_1)\circ (\Psi_0\DT\Id^2)\circ L\\
  M' &= \bigl((\Id\DT H)+C+D\bigr)\circ (\Psi_0\DT\Id^2)\circ K + (\Id\DT\Psi_1)\circ (\Psi_0\DT\Id^2)\circ M\\
  &\qquad\qquad+(E+F)\circ(\Psi_0\DT\Id^2)\circ\Omega.
\end{align*}
Here, $E$ is defined similarly to $A$ and $C$, but for $\Id\DT(\Psi_1\circ(\Id\DT\psi)\circ(\Id\DT\phi))$ versus $(\Id\DT\Phi_1)\circ(\Id^2\DT\psi)\circ(\Id^2\DT\phi)$, and $F$ is similar to $B$ and $D$, but for $\Id\DT(\psi\circ\phi\circ\Psi_1)$ versus $(\Id\DT\psi)\circ(\Id\DT\phi)\circ(\Id\DT\Psi_1)$.
This diagram commutes coherently, in the sense that, if we let $d$ denote the sum of the internal differential at each vertex and all the maps associated to arrows in the diagram, then $d^2=0$. For instance, considering the component of $d^2$ from the top-left $\CFa(Y)$ to the bottom-right $\CFa(Y_0(K))$, this encodes the relation that $\partial\circ M'+M'\circ\partial+H'\circ i +p\circ G'=0$, which in turn follows from the relation 
\[
\bigl((\Id\DT H)+C+D\bigr)\circ(\Id^2\DT\phi)+(\Id\DT\psi)\circ\bigl((\Id\DT G)+A+B\bigr)+\partial\circ (E+F)+(E+F)\circ\partial=0
\]
and Formulas~\eqref{eq:K-htpy} and~\eqref{eq:M-htpy} above.

Replacing each column by its mapping cone, Formula (7.3$'$) exhibits the involutive Floer complex of $Y$ as the mapping cone of a map from the involutive Floer complex of $Y_{-1}(K)$ to the involutive Floer complex of $Y_0(K)$. This implies the result.

\tocless\subsubsection*{Acknowledgements.} We thank Akram Alishahi for pointing out this mistake and for helpful comments on this note.



\end{document}